\documentclass[11pt]{amsart}

 \usepackage{amssymb}
 \usepackage[colorlinks,linkcolor={blue}]{hyperref}
 \usepackage{graphicx}
 \usepackage{amscd}
 \usepackage{mathtools}
 \usepackage{enumerate}
  \usepackage{wasysym} 
 \usepackage[dvipsnames]{xcolor}
 \usepackage{yhmath}
 \usepackage{psfrag}
 \usepackage{scalerel}

 \def\a{\alpha}
 
 \def\be{\beta}
 
 \def\C{{\mathbb C}}
 \def\de{\delta}
 \def\De{\Delta}
 \def\e{\varepsilon}
 \def\deta{{\dot{\eta}}}

 \def\ga{\gamma}
 \def\dga{{\dot{\gamma}}}
 
 \def\tga{{\tilde{\gamma}}}
 \def\Ga{\Gamma}
 
 \def\vr{\varphi}
 \def\vrt{\vartheta}

 \def\la{\lambda}
 
 \def\La{\Lambda}
 \def\bLa{{\mathbf{\Lambda}}}
 
 \def\si{\sigma}
 \def\Si{\Sigma}
 \def\om{\omega}
 \def\Om{\Omega}

 \def\tt{\theta}

 \def\otau{{\ov\tau}}

 \def\re{{\mathbb R}}
 \def\na{{\mathbb N}}

 \def\then{\Longrightarrow}
 \def\ov{\overline}
 \def\Z{{\mathbb Z}}

\def\Á{\textexclamdown}

 \def\A{{\mathbb A}}
 \def\cA{{\mathcal A}}

 \def\B{{\mathbb B}}
 \def\cB{{\mathcal B}}
 
 \def\cC{{\mathcal C}}

 \def\D{{\mathbb D}}

 \def\cE{{\mathcal E}}

 \def\E{{\mathbb E}}

 \def\cH{{\mathcal H}}

 \def\I{{\mathbb I}}

 \def\K{{\mathbb K}}
 
 \def\L{{\mathbb L}}

 \def\cM{{\mathcal M}}

 \def\mN{{\widetilde{\mathcal N}}}

 \def\cO{{\mathcal O}}
 \def\bO{{\mathbb O}}
 \def\cP{{\mathcal P}}

 \def\cS{{\mathcal S}}

 \def\cT{{\mathcal T}}
 \def\fT{{\mathfrak T}}

 \def\cU{{\mathcal U}}

 \def\ou{{\ov{u}}}

 \def\cV{{\mathcal V}}

 \def\cW{{\mathcal W}}
 \def\ow{{\ov{w}}}

 \def\tq{{\tilde{q}}}
 
 \def\tq1{{\tilde{q}_1}}

 \def\ox{{\ov{x}}}

 \def\X{{\mathbb X}}
 \def\fX{{\mathfrak X}}

 \def\oy{{\ov{y}}}

 \def\oz{{\ov{z}}}

 \def\0{{\mathbf 0}}

 \def \lv{\left\vert}
 \def \rv{\right\vert}
 \def \lV{\left\Vert}
 \def \rV{\right\Vert}
 \def \ov{\overline}
 
 \def \then{\Longrightarrow}

 \definecolor{dgreen}{rgb}{0,0.3,0}
 \definecolor{dred}{rgb}{0.8,0,0}

 \def\ee{\text{\rm\large  e}}

 \DeclareMathOperator*{\tsum}{{\textstyle \sum}}
 \DeclareMathOperator{\supp}{supp}
 
 \DeclareMathOperator{\diam}{diam}

 \DeclareMathOperator{\intt}{int}

 \DeclareMathOperator{\Lip}{Lip}

  \renewcommand{\proofname}{{\bf Proof:}}

 \theoremstyle{plain}

 \newtheorem{MainThm}{Theorem}
 
 \newtheorem{MainProp}[MainThm]{Proposition}

 \newtheoremstyle{Cl}
  {5pt}
  {3pt}
  {\sl}
  {}
  {\it}
  {:}
  {.5em}
  {}

 \newtheoremstyle{St}
  {5pt}
  {3pt}
  {\sl}
  {}
  {\bf}
  {.}
 {\newline}
  {}

 \def\begincproof{
                  \renewcommand{\proofname}{\it Proof:}
                  \begin{proof}
                 }

 \def\endcproof{
                \renewcommand{\qedsymbol}{$\diamondsuit$}
                \end{proof}
                \renewcommand{\qedsymbol}{\openbox}
                \renewcommand{\proofname}{\bf Proof:}
               }

  \swapnumbers

 \newtheorem{Thm}{Theorem}[section]
 
 \newtheorem{Lemma}[Thm]{\bf Lemma}
 
 \newtheorem{Corollary}[Thm]{\bf Corollary}
 \newtheorem{Theorem}[Thm]{\bf Theorem}
 \newtheorem{Proposition}[Thm]{\bf Proposition}

\swapnumbers
 \theoremstyle{Cl}
 \newtheorem{Claim}{Claim} [Thm]

\swapnumbers

 \theoremstyle{St}

 \theoremstyle{definition}
 \newtheorem{Definition}[Thm]{\bf Definition}

 \theoremstyle{remark}
 \newtheorem{Remark}[Thm]{\bf Remark}
 
 \newtheorem{mysec}[Thm]{}

\usepackage[colorinlistoftodos]{todonotes}
\makeatletter
\providecommand\@dotsep{5}
\makeatother
 \renewcommand{\proofname}{{\bf Proof:}}

 \title[zero entropy.]
 {Generic hyperbolic Ma\~n\'e sets have zero entropy.}



 \author[G. Contreras]{Gonzalo Contreras}

\address{CIMAT \\
          A.P. 402, 36.000 \\
          Guanajuato. GTO \\
          M\'exico.}
\email{gonzalo@cimat.mx}
\thanks{Partially supported by CONACYT, Mexico, grant 
A1-S-10145.}

\begin{document}

\makeatother

\parskip +5pt

\begin{abstract}\quad

We prove that in the $C^k$ topology generic hyperbolic Ma\~n\'e sets have zero topological entropy.
\end{abstract}

\maketitle

\tableofcontents


Let $M$ be a closed riemannian manifold.
A Tonelli Lagrangian is a $C^2$ function
\linebreak
 $L:TM\to\re$ that  is 
\begin{enumerate}[(i)]
\openup 5pt
\item {\it Convex:} $\exists a>0$\;
$\forall (x,v), (x,w)\in TM$, \;
$w\cdot \partial^2_{vv} L(x,v)\cdot w \ge a |w|_x^2$.
\end{enumerate}
\pagebreak
The uniform convexity assumption and the compactness of $M$ imply that $L$ is
\begin{enumerate}[(i)]
\setcounter{enumi}{1}
\item {\it Superlinear:}
$\forall A>0$ $\exists B>0$ such that
$\forall (x,v)\in TM$: $L(x,v)> A\,|v|_x-B$.
\end{enumerate}

Given $k\in\re$, the Ma\~n\'e {\it action potential} is defined as
$\Phi_k:M\times M\to\re\cup\{-\infty\}$, 
\begin{equation}\label{actionpotential}
\Phi_k(x,y):=\inf_{\ga\in \cC(x,y)}
\int k+L(\ga,\dga),
\end{equation}
where 
\begin{equation}\label{defcxy}
\cC(x,y):=\{\ga:[0,T]\to M\; {\text{absolutely continuous }}\vert\;
T>0,\; \ga(0)=x,\;\ga(T)=y\;\}.
\end{equation}
The  Ma\~n\'e {\it critical value} is
\begin{equation}\label{defcL}
c(L) := \sup\{\, k\in \re \;|\; \exists\, x,y\in M :\;\Phi_k(x,y)=-\infty\;\}.
\end{equation}
See  \cite{CILib} for several characterizations of $c(L)$.

A curve $\ga:\re\to M$ is {\it semi-static} if 
$$
\forall s<t  \qquad
\int_s^t c(L)+L(\ga,\dga) = \Phi_{c(L)}(\ga(s),\ga(t)).
$$
Also $\ga:\re\to M$ is {\it static} if
$$
\forall s<t  \qquad
\int_s^t c(L)+L(\ga,\dga) = -\Phi_{c(L)}(\ga(t),\ga(s)).
$$
The {\it Ma\~n\'e set} of $L$ is
$$
\mN(L):=\{(\ga(t),\dga(t))\in TM \;|\; t\in\re,\; \ga:\re\to M\text{ is semi-static }\},
$$
and the {\it Aubry set } is
$$
\cA(L):=\{(\ga(t),\dga(t))\in TM \;|\; t\in\re,\; \ga:\re\to M\text{ is static }\}.
$$

The Euler-Lagrange equation
$$
\tfrac d{dt}\, \partial_vL = \partial_x L
$$
defines the Lagrangian flow $\vr_t$ on $TM$.
The {\it energy function} $E:TM\to\re$,
$$
E(x,v):= \partial_v L(x,v)\cdot v -L(x,v),
$$
is invariant under the Lagrangian flow.
The Ma\~n\'e set $\mN(L)$ is invariant under the Lagrangian flow and
it is contained in the energy level $\cE:=[E=c(L)]$ (see e.g. Ma\~n\'e \cite[p.~146]{Ma7} or \cite{CILib}).

Let $\cM_{\text{inv}}(L)$ be the set of Borel probabilities in $TM$ which are invariant under the Lagrangian flow.
Define the {\it action } functional   $A_L:\cM_{\text{inv}}(L)\to\re\cup\{+\infty\}$ as 
$$
A_L(\mu):=\int L\,d\mu.
$$
The set of {\it minimizing measures} is
$$
\cM_{\min}(L):=\arg\min_{\cM_{\text{inv}}(L)} A_L,
$$
and the {\it Mather set } $\cM(L)$ is the union of the support of minimizing measures:
$$
\cM(L):=\bigcup_{\mu\in\cM_{\min}(L)} \supp(\mu).
$$
Ma\~n\'e proves (cf. Ma\~n\'e~\cite[Thm. IV]{Ma7} also \cite[p. 165]{CDI}) that an invariant measure is minimizing if and only if it is supported in the Aubry set. Therefore we get the set of inclusions
\begin{equation}\label{mane}
\cM\subseteq\cA\subseteq\mN\subseteq\cE.
\end{equation}

\begin{Definition}\label{defhipL}\quad

We say that $\mN(L)$ is {\it hyperbolic} if there are 
sub-bundles\footnote{The continuity of the bundles $E^s$, $E^u$ is a consequence of Definition~\ref{defhipL} 
and it is not required in \ref{defhipL}, see
Proposition 5.1.4 in \cite{FH}.}
$E^s$, $E^u$ of $T\cE|_{\mN(L)}$ and $T_0>0$ such that
\begin{enumerate}[(i)]
\item $T\cE|_{\mN(L)} = E^s \oplus\langle \frac d{dt}\vr_t\rangle \oplus E^u$.
\item $\lV D\vr_{T_0}|_{E^s}\rV < 1$, $\lV D\vr_{-T_0}|_{E^u}\rV <1$.
\item $\forall t\in\re$ \quad $(D\vr_t)^*(E^s)=E^s$, $(D\vr_t)^*(E^u)=E^u$.
\end{enumerate}
\end{Definition}

 Hyperbolicity for {\sl autonomous}
lagrangian or hamiltonian flows is  always understood as hyperbolicity
for the flow restricted to the energy level.

Fix a Tonelli Lagrangian $L_0$. Let
$$
\cH^k(L_0):=\{\,\phi\in C^k(M,\re)\;|\; \mN(L_0+\phi) \text{ is hyperbolic }\},
$$
endowed with the $C^k$ topology.
By \cite[Lemma 5.2, p. 661]{CP} the map $\phi\mapsto \mN(L_0+\phi)$ is upper
semi-continuous. Therefore $\cH^k(L_0)$ is an open set for any $k\ge 2$.

For possible applications we want to remark that all the perturbing 
potentials in this paper are locally of the
form 
\begin{equation}\label{canal}
\phi(x) = \e\, d(x,\pi(\Ga))^k, 
\quad k\ge 2,
\end{equation}
 where $\Ga$ is a 
suitably chosen periodic orbit of the Lagrangian flow nearby the Ma\~n\'e set.
Here we prove

  \begin{MainThm}\label{Tzeroentropy}
If $L_0$ is a Tonelli lagrangian and $k\ge 2$, 
the set
$$
\cE_0(L_0) =\{\,\phi\in\cH^k(L_0)\;|\;
\mN(L_0+\phi)\text{ has zero topological entropy }\}
$$
contains a residual subset of $\cH^k(L_0)$.
\end{MainThm}

On surfaces Theorem~\ref{Tzeroentropy} follows from
the graph property of the Aubry set, namely

\begin{MainProp}
If $L$ is a Tonelli Lagrangian, $k\ge 2$ and $\dim M=2$, \\
then
$\mN(L)$ has zero topological entropy.
\end{MainProp}
\begin{proof}
The topological entropy of $\mN(L)$ is the supremum of the metric entropies of 
the invariant measures supported on $\mN(L)$, see Theorem~\ref{VarPrin}.  
By \cite[Theorem V(c)]{Ma7} the $\om$-limit of any orbit in $\mN(L)$ is in the Aubry set $\cA(L)$.
Thus, by the Poincar\'e Recurrence Theorem, any invariant measure on $\mN(L)$ is supported on $\cA(L)$.
Therefore it is enough to
prove that the Aubry set has zero topological entropy. 
By the Graph Property \cite[Theorem VI(b)]{Ma7}, the flow on $\cA(L)$ is the image   
under a Lipschitz (conjugacy) map $(\pi|_{\cA(L)})^{-1}$
of a flow on a (Lipschitz) continuous  lamination on the surface $M$.
By Fathi~\cite[Lemma~3.3]{Fa7} and Young~\cite{lsy1} the flow on the 
projected Aubry set $\pi(\cA(L))$ has zero entropy and then 
(Walters~\cite[Theorem~7.2]{Walters}) $\cA(L)$ has zero topological entropy.
\end{proof}

The proof of Theorem~\ref{Tzeroentropy} is done in two steps. 
 The first step in subsection~\ref{ssappperorb} is the study of 
 how well hyperbolic minimizing measures can be approximated 
 by closed  orbits of a given period. For approximation with large periods
 on generic lagrangians see Ma\~n\'e~\cite[Theorem~F]{Ma6}.
 We found that the arguments of this proof follow elegantly using
 symbolic dynamics on the Ma\~n\'e set.

 For the second part of the proof of Theorem~\ref{Tzeroentropy}
 we show that for $\ga>0$ the set 
 $$
 \cT_\ga:=\{\,\phi\in C^k(M,\re)\;|\;
 h_{top}(\mN(L+\phi))\le \ga\,\}
 $$
 contains an open and dense set in $C^k(M,\re)$.
 For the open part we use the upper semicontinuity of the 
 Ma\~n\'e set and the uniform upper semicontinuity of the 
 metric entropy for h-expansive maps, which we prove in
 Appendix~\ref{AE}. The uniform h-expansivity needed is
 proved in Appendix~\ref{asha}. For the density we use 
 a short closed orbit with small action obtained in the first 
 step, perturb the lagrangian with a canal as in~\eqref{canal}
 and show that the new minimizing measures have to accumulate
 nearby the periodic orbit. Then we show that the entropy nearby
 a short periodic orbit must be small.
 
  In order to use structural stability of hyperbolic flows, so we can fix
  the symbolic dynamics for nearby flows we need aproximation of the
  vector fields in the $C^1$ topology. All the proofs of structural stability 
  in the literature are made for hyperbolic flows on a fixed manifold. 
  In our case the manifold is the energy level $\E_\phi:=E^{-1}(c(L+\phi))$
  which varies with the potential $\phi$. If the lagrangian flow in $\E_\phi$ 
  have no singularities these energy levels can be identified with the unit 
  tangent bundle, see Corolary~\ref{CSH1}. But $\E_\phi$ 
  depends on $c(L+\phi)$ which varies continuously on $\phi$.
  Appendices~\ref{asst} and \ref{ashms} are included to address this 
  apparent impasse. Ending with the desired statement in Remark~\ref{remell}.
  We need this for persistence of hyperbolicity, uniform constants in the shadowing lemma, 
  and the uniform h-expansivity.

    Since we assume hyperbolicity of the Aubry set, which is chain recurrent
    (cf. Ma\~n\'e~\cite[Theorem V]{Ma7}), the shadowing lemma implies the 
    existence of many periodic orbits nearby. Along this paper we do not perturb
    recurrent orbits to close them. We just choose carefully an existing periodic
    orbit and perturb the lagrangian by a canal as in~\eqref{canal}.
    
    \bigskip

    \bigskip
    
    \color{black}

\section{The Aubry set.}\quad

Recall that a curve $\ga:\re\to M$ is {\it static} for a Tonelli Lagrangian $L$
if 
$$
s<t \quad \then\quad
\int_s^t c(L)+ L(\ga,\dga)  = -\Phi_{c(L)}(\ga(t),\ga(s));
$$
equivalently (cf. Ma\~n\'e \cite[pp. 142--143]{Ma7}), 
if $\ga$ is semi-static and 
\begin{equation}\label{aubry}
 s<t\quad\then\quad \Phi_{c(L)}(\ga(s),\ga(t))+\Phi_{c(L)}(\ga(t),\ga(s))=0.
\end{equation}
The {\it Aubry set} is defined as
$$
\cA(L):=\{\,(\ga(t),\dga(t))\;|\; t\in\re, \;\ga\text{ is static}\;\},
$$
its elements are called {\it static vectors}.

\begin{Lemma}[A priori bound]\label{priori}\quad

For $C>0$ there exists $A_0=A_0[C]>0$ such that 
if $\ga:[0,T]\to M$
is a solution of the Euler-Lagrange equation with $A_L(\ga)<C\, T$, then 
$$
\lv\dga(t)\rv < A_0\qquad\text{ for all }t\in[0,T].
$$
\end{Lemma}

\begin{proof}
The Euler-Lagrange flow preserves the {\it energy function} 
\begin{equation}\label{EL}
E_L:= v \cdot \partial_v L -L.
\end{equation}
We have that
\begin{align}
\hskip -0.7cm\forall s\ge0\qquad
\tfrac{d\,}{ds}E_L(x,sv)\big|_s&=s\, v\cdot \partial_{vv}L(x,v)\cdot v \ge s\, a |v|_x^2.
\notag\\
E_L(x,v) &= E_L(x,0)+\int_0^{1} \tfrac{d\,}{ds}E_L(x,s v) \, ds
\notag\\
&\ge \min_{x\in M}E_L(x,0)+ \tfrac 12 a |v|_x^2.
\label{lbel}
\end{align}
Let
$$
g(r):= \sup\big\{w\cdot \partial_{vv}L(x,v)\cdot w \,:\, |v|_x\le r,\,|w|_x= 1\big\}.
$$
Then $g(r)\ge a$ and 
\begin{equation}\label{ubel}
E_L(x,v)\le \max_{x\in M}E_L(x,0) + \tfrac 12\, g(|v|_x)\, |v|_x^2.
\end{equation}

By the superlinearity there is $B>0$ such that $L(x,v)>|v|_x-B$ for all $(x,v)\in TM$.
Since $A_L(\ga)<C\,T$, the mean value theorem implies that there is $t_0\in]0,T[$ 
such that $|\dga(t_0)|<B+C$. Then~\eqref{ubel} gives an upper bound 
on the energy of $\ga$ and~\eqref{lbel} bounds the speed of $\ga$.

\end{proof}

 For $x,\,y\in M$ and $T>0$ define
 \begin{equation*}\label{defctxy}
 \cC_T(x,y):=
 \{\,\ga:[0,T]\to M\,|\, \ga \text{ is absolutely continuous}, \ga(0)=x,\,\ga(T)=y\,\}.
 \end{equation*}

 \begin{Corollary}\label{CAPB}\quad
 
 There exists $A_1>0$ such that if  $x,\,y\in M$ and $\ga\in\cC_T(x,y)$ is
 a solution of the Euler-Lagrange equation with 
 \begin{equation}\label{aphm}
 A_{L+c}(\ga) \le \Phi_c(x,y) + \max\{T,d(x,y)\},
  \end{equation}
 where $c=c(L)$,
 then
 \begin{enumerate}[(a)]
 \item $T\,\ge\,\tfrac 1{A_1}\; d(x,y)$.
 \label{CAPB1}
 \item  \;$|\dga(t)|\,\le\,A_1$ \; for all $t\in[0,T]$.
 \label{CAPB2}
 \end{enumerate}
 \end{Corollary}
 
 \begin{proof}
 First suppose that  $d(x,y)\le T$. Then item~\eqref{CAPB1} holds with $A_1=1$.
 Let
 \begin{equation}\label{llr}
 \ell(r):=|c|+\sup\{\,L(x,v)\,|\, (x,v)\in TM,\, |v|\le r\,\}.
 \end{equation}
 Since $d(x,y)\le T$, there exists a $C^1$ curve $\eta:[0,T]\to M$  joining $x$ to $y$ with $|\deta|\le 1$.
 Using~\eqref{aphm}, we have that 
 $$
 A_{L+c}(\ga)\le \Phi_c(x,y)+T\le A_{L+c}(\eta)+T
 \le \big(\ell(1)+c\big)\,T+ T.
 $$
 Then item~\eqref{CAPB2} holds for $A_1=A_0[\ell(1)+c+1]$ where $A_0$ is from Lemma~\ref{priori}.
 
 Now suppose that $d(x,y)\ge T$.
 Let $\eta:[0,d(x,y)]\to M$ be a minimal geodesic with 
 $|\deta|\equiv 1$ joining $x$ to $y$. Let 
 $D:=\ell(1)+c +2>0$. From the superlinearity property there is 
 $B>1$ such that
 $$
 L(x,v) + c > D\, |v| - B, \qquad \forall (x,v)\in TM.
 $$
 Then
 \begin{align}
 [\ell(1)+c]\;d(x,y) 
    &\ge A_{L+c}(\eta) \ge \Phi_c(x,y)         
    \label{ap-1}\\
    &\ge A_{L+c}(\ga)-d(x,y) \quad\qquad \text{using \eqref{aphm}}
    \label{ap-2}\\
    &\ge\int_0^T\bigl( D\;|\dga|-B\,\bigr)\,dt - d(x,y)
    \notag\\
    &\ge D\;d(x,y) - B\,T - d(x,y).
    \notag
 \end{align}
 Hence
 $$
 T \ge \tfrac{D-\ell(1)-c-1}{B}\; d(x,y) = \tfrac 1B \; d(x,y).
 $$
 From~\eqref{ap-1} and~\eqref{ap-2}, we get that
 \begin{align*}
 A_L(\ga)&\le \bigl[\,\ell(1)+c+1\,\bigr]\;d(x,y)-c\,T ,
        \\
         &\le \bigl\{\,B\,[\,\ell(1)+c+1\,]-c\,\bigr\}\;T.
 \end{align*}
 Since $B>1$, Lemma~\ref{priori} completes the proof.
 
 \end{proof}

 We say that a curve $\ga:[0,T]\to M$ is a {\it Tonelli minimizer} if it minimizes the action functional
 on $\cC_T(\ga(0),\ga(T))$, i.e. if it is a minimizer with fixed endpoints and fixed time interval.

 \begin{Corollary}\label{Cbtm}
 There is $A>0$ such that if $x,\,y\in M$ and $\eta_n\in\cC_{T_n}(x,y)$, $n\in\na^+$ is a Tonelli minimizer
 with
 $$
 A_{L+c}(\eta_n)\le \Phi_c(x,y)+\tfrac1n,
 $$
 then there is $N_0>0$ such that $\forall n>N_0$, $\forall t\in[0,T_n]$, $|\deta_n(t)|<A$.
 \end{Corollary}

 \begin{proof}
 If $d(x,y)>0$ then for $n$ large enough $d(x,y)>\tfrac 1n$.
 In this case Corollary~\ref{CAPB} implies the result with the constant $A_1$.
 If $d(x,y)=0$ let $\xi_n:[0,T_n]\to\{x\}$ be the constant curve.
 Since $\eta_n$ is a Tonelli minimizer, we have that
 $$
 A_L(\eta_n)\le A_L(\xi_n) =\int_0^{T_n}L(x,0)\,dt\le |L(x,0)|\, T_n.
 $$
 Lemma~\ref{priori} implies that $|\deta_n|\le A_0(C)$ with $C=1+\sup_{x\in M}|L(x,0)|$.
 Now take 
 \linebreak
 $A=\max\{A_0[C],A_1\}$.
 \end{proof}

\begin{Lemma}\label{ALinv}\quad

If $(x,v)$ is a static vector then $\ga:\re\to M$, $\ga(t)=\pi\vr_t(x,v)$ 
is a static curve, i.e. the Aubry set $\cA(L)$ is invariant.
\end{Lemma}

 \begin{proof}\quad
 
 Let $\ga(t)=\pi\,\vr_t(x,v)$ and suppose that $\ga|_{[a,b]}$ is static.
 We have to prove that all $\ga|_\re$ is static. 
 Let
 $\eta_n\in\cC_{T_n}(\ga(b),\ga(a))$
 be a Tonelli minimizer  with
  $$
 A_{L+c}(\eta_n)<\Phi_c(\ga(b),\ga(a))+\tfrac 1n.
 $$
 
 By Corollary~\ref{Cbtm}, for $n$ large enough,
 $|\deta_n|<A$. We can assume that 
 $\deta_n(0)\to w$. Let $\xi(s)=\pi\,\vr_s(w)$. If $w\ne \dga(b)$
 then for some $\e>0$ the curve $\ga|_{[b-\e,b]}*\xi|_{[0,\e]}$ is not $C^1$, and
 hence  it
 can not be a (Tonelli) minimizer of $A_{L+c}$ in 
 $\cC_{2\e}\big(\ga(b-\e),\xi(\e)\big)$. Thus
 $$
 \Phi_c(\ga(b-\e),\xi(\e))
        < A_{L+c}(\ga|_{[b-\e,b]})
        + A_{L+c}(\xi|_{[0,\e]}).
 $$
 \begin{align*}
 &\Phi_c(\ga(a),\ga(a))
    \le \Phi_c(\ga(a),\ga(b-\e))
         +\Phi_c(\ga(b-\e),\xi(\e))
         +\Phi_c(\xi(\e),\ga(a))
    \\
    &\;< A_{L+c}(\ga_{[a,b-\e]})
       +A_{L+c}(\ga|_{[b-\e,b]})
       +A_{L+c}(\xi|_{[0,\e]})
       +\liminf_n A_{L+c}(\eta_n|_{[\e,T_n]})
    \\
    &\;\le A_{L+c}(\ga|_{[a,b]})
      +\lim_nA_{L+c}\bigl(\,\eta_n|_{[0,\e]}*\eta_n|_{[\e,T_n]}\bigr)       
    \\
    &\;= -\Phi_c(\ga(b),\ga(a))+\Phi_c(\ga(b),\ga(a))
    =0.
 \end{align*}
 Thus there is a closed curve, from $\ga(a)$ to itself, with negative
 $L+c$ action, and also negative $L+k$ action for some $k>c(L)$.  
 Concatenating the curve with itself many times  
 shows that $\Phi_k(\ga(a),\ga(a))=-\infty$.
 By~\eqref{defcL} this implies that $k\le c(L)$, which is 
 a contradiction.
 Thus $w=\dga(b)$ and similarly $\lim_n\deta_n(T_n)=\dga(a)$.
 
 If $\limsup T_n<+\infty$, we can assume that $\tau=\lim_n T_n>0$
 exists. In this case $\ga$ is a semi-static
 periodic orbit of period $\tau+b-a$ and then $\ga|_\re$  is static.
 
 Now suppose that $\lim_n T_n=+\infty$.
 If $s>0$, we have that
 \begin{align*}
 A_{L+c}&(\ga|_{[a-s,b+s]})
   +\Phi_c(\ga(b+s),\ga(a-s))\le
   \\
   &\le\lim_n\big\{\,A_{L+c}(\eta_n|_{[T_n-s,T_n]})
       +A_{L+c}(\ga|_{[a,b]})
       \begin{aligned}[t]
       &+A_{L+c}(\eta_n|_{[0,s]})\,\big\}\\
       &+\Phi_c(\ga(b+s),\ga(a-s))
       \end{aligned}
    \\
    &\le\begin{aligned}[t]
    &-\Phi_c(\ga(b),\ga(a)) \\
    &+\lim_n \big\{\,A_{L+c}(\eta_n|_{[0,s]})
                    +A_{L+c}(\eta_n|_{[s,T_n-s]})
                    +A_{L+c}(\eta_n|_{[T_n-s,T_n]})
             \,\big\}
    \end{aligned}         
    \\
    &\le -\Phi_c(\ga(b),\ga(a))+\Phi_c(\ga(b),\ga(a))              
    =0.
 \end{align*}
 Thus $\ga_{[a-s,b+s]}$ is static for all $s>0$.

\end{proof}

Let  $\cM_{\text{inv}}(L)$ be the set of Borel probabilities in $TM$ invariant under the 
 Lagrangian flow.
Denote by $\cM_{\min}(L)$ the set of minimizing measures for the Lagrangian $L$, i.e.
\begin{equation}\label{Mmin}
\cM_{\min}(L):=\Big\{\,\mu\in\cM_{\text{inv}}(L)\;\Big|\;
\int_{TM}L\,d\mu = -c(L)\;\Big\}.
\end{equation}
Their name is justified (cf. Ma\~n\'e \cite[Theorem II]{Ma7}) by 
\begin{equation}\label{minmeas}
-c(L) = \min_{\mu\in\cM_{\text{inv}}(L)}\int_{TM} L \; d\mu 
= \min_{\mu\in\cC(TM)}\int_{TM} L \; d\mu.
\end{equation}
 Fathi and Siconolfi \cite[Theorem 1.6]{FaSi} prove the second equality in \eqref{minmeas}
 where the set of  {\it closed measures} is defined by 
$$
{\mathcal C}(TM):=\Big\{ \, \mu \text{ Borel probability on }TM\; \Big\vert\;
\forall \phi\in C^1(M,\re) \; \int_{TM} d\phi\; d\mu =0\,\Big\}.
$$

Recall
\begin{Theorem}[Theorem~IV in \cite{Ma7} or~\cite{CDI}]\label{ecmm}\quad

A probability $\mu\in\cM_{\text{inv}}(L)$ is minimizing if and only if 
$\supp\mu\subset\cA(L)$.
\end{Theorem}

\begin{Corollary}\label{cecmm}\quad

A probability $\mu\in\cM_{\text{inv}}(L)$ is minimizing if and only if 
$\supp\mu\subset\mN(L)$.
\end{Corollary}

\begin{proof}
It is enough to prove that invariant probabilities supported in $\mN(L)$ are 
actually supported in $\cA(L)$.
Denote by $\vr_t=\vr^L_t$  the Lagrangian flow.
We first prove that the non-wandering set of the restriction $\vr_t|_{\mN(L)}$
satisfies $\Om(\vr_t|_{\mN(L)})\subset\cA(L)$. If $\vrt\in\Om(\vr^L_t|_{\mN(L)})$ then there is a sequence $\tt_n\in\mN(L)$ and $t_n\ge 2$ such that $\lim_n\tt_n=\vrt=\lim_n\vr_{t_n}(\tt_n)$. 
The action potential $\Phi_k$ in~\eqref{actionpotential} is Lipschitz 
by Theorem~I in Ma\~n\'e~\cite{Ma7} or~\cite{CDI}.
Then
\begin{align*}
A_{L+c}(\vr_{[0,1]}(\vrt))+\Phi_{c(L)}\big(\pi\vr_1(\vrt),\pi(\vrt)\big)
&\le \lim_n A_{L+c}(\vr_{[0,1]}(\tt_n))+ \lim_n A_{L+c}(\vr_{[1,t_n]}(\tt_n))
\\
&\le \lim_n A_{L+c}(\vr_{[0,t_n]}(\tt_n)) =\lim_n\Phi_{c(L)}(\pi\tt_n,\pi\vr_{t_n}(\tt_n))
\\
&\le \Phi_{c(L)}(\pi\vrt,\pi\vrt)=0.
\end{align*}
And hence $\vrt\in\cA(L)$.
Now, if $\mu\in\cM_{\text{inv}}(L)$ has $\supp\mu\subset \mN(L)$, by Poincar\'e recurrence theorem
$\supp\mu\subset \Om(\vr|_{\mN(L)})\subset\cA(L)$.

\end{proof}

\section{Symbolic Dynamics for the Aubry set.}\quad

Throughout the rest of the section we will identify a periodic orbit with the 
invariant probability supported on
the periodic orbit. 

The first two results, Lemma~\ref{Lper} and Proposition~\ref{Pper} follow
arguments by X. Bressaud and A. Quas~\cite{BQ}.

Let $A\in\{0,1\}^{M\times M}$ be a $M\times M$ matrix  with entries in $\{0,1\}$.
The subshift  of finite type $\Si_A$ associated to $A$ is the set
$$
\Si_A=\big\{\;\ox=(x_i)_{i\in\Z}\in \{1,\ldots,M\}^\Z \;\big|\quad \forall i\in\Z\quad A(x_i,x_{i+1})=1\;\big\},
$$
endowed with the metric
\begin{equation}\label{da}
d_a(\ox,\oy) = a^{-i}, \qquad i=\max\{\;k\in\na\;|\; x_i=y_i\;\;\forall |i|\le k\;\}
\end{equation}
for some $a>1$ and the {\it shift transformation} is
$$
\si:\Si_A\to\Si_A, \qquad \forall i\in\Z\quad \si(\ox)_i = x_{i+1}.
$$

\bigskip

\begin{Lemma}\label{Lper}
Let $\Si_A$ be a shift of finite type  with $M$ symbols and topological
entropy $h$. 
Then $\Si_A$ contains a periodic orbit of period at most $1+M \ee^{1-h}$.
\end{Lemma}

\begin{proof}
Let $k+1$ be the period of the shortest periodic orbit in $\Si_A$.
We claim that a word of length $k$ in $\Si_A$ is determined by the set
of symbols that it contains.
First note that since there are no periodic orbits of period $k$ or less,
any allowed $k$-word must contain $k$ distinct symbols.
Now suppose that $u$ and $v$ are two distinct words of length $k$ 
in $\Si_A$ containing the same symbols. Then, since the words are different, 
there is a consecutive pair of   symbols, say $a$ and $b$, in $v$ which 
occur in the opposite order (not necessarily consecutively) in $u$.
Then the infinite concatenation of the segment of $u$ starting at $b$ 
and ending at $a$  gives a word in $\Si_A$ of period at most $k$, which 
contradicts the choice of $k$.

It follows that there are at most $\binom Mk$ words of length $k$.
Using the basic properties of topological entropy
\begin{align*}
\ee^{hk}\le \binom Mk\le \frac {M^k}{k!}
\le \left(\frac {M \ee} k\right)^k.
\end{align*}
Taking $k$th roots, we see that $k\le M\ee^{1-h}$.

\end{proof}

Theorem~6.6.5 in \cite{FH} implies  that if $\cA(L)$ is hyperbolic then there is a
hyperbolic set $\La$ in the energy level $c(L)$ which contains $\cA(L)$,
$\cA(L)\subset \La\subset E^{-1}\{c(L)\}$, and which has a Markov partition,
as defined in  \cite[Def. 6.6.4]{FH}.
The Markov partition consists of a finite 
set of rectangles $\fT=\{T_i\}_{i=1}^r$ included in mutually disjoint transversal disks 
$D_i\supset T_i$. There is a Lipschitz first return time 
$\tau:\cup\fT\to]0,\a]$, $\cup\fT:=\cup_{T\in\fT}\, T$, 
$$
\tau(\tt):=\min\{\,t>0\;|\; \vr_t(\tt)\in\cup\fT\,\},
$$
and first return map (also called Poincar\'e map) $F:\cup\fT \to\cup\fT$,
$$
F(\tt):=\vr_{\tau(\tt)}(\tt).
$$
By Theorem~6.6.5 in \cite{FH}
there is a subshift of finite type $\Om$ and a Lipschitz map
$\Pi:\Om\to \cup\fT$ which is a semiconjugacy between
the shift map $\si$ and the Poincar\'e map $F$.
Also $\Pi$ extends to a time preserving Lipschitz semiconjugacy
$\Pi:S(\Om,\otau)\to \La$ from the suspension $S(\Om,\otau)$ 
of the shift with return time $\otau:=\tau\circ\Pi:\Om\to]0,\a]$, 
as defined in~\eqref{som} in Appendix~\ref{ASD},
 with suspended flow $S_t$
defined in~\eqref{sust}, to the lagrangian flow $\vr_t|_\La$ on $\La$.
In other words, the following diagram commutes for every $t\in\re$ and
$\Pi$ is Lipschitz.
$$
\begin{CD}
S(\Om,\otau) @>S_t>> S(\Om,\otau)
\\
@V\Pi VV   @VV \Pi V
\\
\La @> \vr_t >> \La
\end{CD}
$$

\openup -1pt

A $\si$-invariant measure $\nu$ on $\Om$ induces a $\vr_t$-invariant
measure $\mu_\nu$ on $\La$ by
$$
\int_\La f\,d\mu_\nu :=
\int_\Om\left[\int_0^{\otau(\ow)} f\big(\vr_t(\Pi(\ow))\big)\,dt\right]d\nu(\ow).
$$
Define $B:\Om\to\re$ by 
$$
B(\ow):=\int_0^{\otau(\ow)} \Big[L\big(\vr_t(\Pi(\ow))\big)+c(L)\Big]\,dt.
$$
Then we have that
$$
A_{L+c(L)}(\mu_\nu)=\int_\La (L+c(L))\,d\mu_\nu =\int_\Om B\, d\nu.
$$
To fix notation we use the metric $d_a$ from~\eqref{da} on $\Om$, namely
\begin{equation*}
d_a(\ou,\ow):=a^{-n},
\qquad  
n:=\max\{\,k\in\na\;:\;\forall|i|\le k,\quad u_i=w_i\,\},
\end{equation*}
and $a>1$ is chosen  such that $\Pi:\Om\to \cup\fT$ 
is Lipschitz. Indeed, by Gr\"onwall's inequality the image on an n-cylinder
contains a ball of radius $A\, (\Lip \vr_B)^{-n}$ where $A$ is constant and
$B=\sup \tau$. So if $a>1$ is large enough, then $\Pi$ is Lipschitz.

The subshift $\Om$ has
symbols in $\fT$ and a transition matrix $A:\fT\times\fT\to\{0,1\}$ such that
$$
\Om =\Si(A)=\{\,\ow\in \fT^\Z\;:\; \forall i\in\Z\quad A(w_i,w_{i+1})=1\,\}.
$$
We say that $(w_1,\ldots,w_n)\in\fT^n$ is a {\it legal word in $\Om$} iff
$A(w_i,w_{i+1})=1$ for all $i=1,\ldots,n-1$.
A {\it cylinder} in $\Om$ is a set of the form
$$
C(n,m,\ow):=\{\,\oz\in\Om : z_i=w_i \quad \forall i=n,\ldots,m\,\},
$$
where $\ow\in\Om$ or $\ow=(w_n,\ldots,w_m)$ is a legal word in $\Om$.
A {\it subshift} of $\Om$ is  a closed $\si$-invariant subset of $\Om$.
If $Y\subset\Om$ is a subshift of $\Om$ we say that  
$\ow=(w_n,\ldots,w_m)$ is a legal word in $Y$ iff
 $C(n,m,\ow)\cap Y\ne\emptyset$. 
  
 Since the diagram
 $$
 \begin{CD}
 \Om @> \si >> \Om
 \\
 @V \Pi VV @VV \Pi V
 \\
 \cup\fT @> F >> \cup\fT
 \end{CD}
 $$
 commutes and $\Pi$ is continuous, the set 
 \begin{equation}
\label{subshY}
 Y:=\Pi^{-1}\big((\cup\fT)\cap\cA(L)\big)
  \end{equation}
  is a subshift of $\Om$ and
  \begin{equation}\label{YAL}
 \begin{CD}
 Y @> \si >> Y
 \\
 @V \Pi VV @VV \Pi V
 \\
 (\cup\fT)\cap\cA(L) @> F >> (\cup\fT)\cap \cA(L)
 \end{CD}  
  \end{equation}
 commutes.

\openup+1pt

\section{Approximation by periodic orbits.}\quad
\label{ssappperorb}

Let $\cP_L(T)$ be the set of  Borel invariant probabilities for $L$ which are
supported on a periodic orbit with period $\le T$. For $\mu\in\cP_L(T)$ 
write
$$
c(\mu,\cA(L)) :=\sup_{\theta\in\supp(\mu)} d(\theta,\cA(L)).
$$

\medskip

\begin{Proposition}\label{Pper}
Suppose that the Aubry set $\cA(L)$ is hyperbolic.
Then for all $\ell\in \na^+$ 
$$
\liminf_{T\to+\infty} \;T^\ell \left(\inf_{\mu\in\cP_L(T)} c(\mu,\cA(L))\right) = 0.
$$
\end{Proposition}    

\begin{proof}\quad

For $n\in\na^+$ let $Z^{(n)}$ be the 1-step subshift of finite type
whose symbols are the legal words of size $n$ in $Y:=\Pi^{-1}(\cA(L))$.
A transition from the word 
\begin{equation}\label{concate}
\text{$u$ to $v$ is allowed in $Z^{(n)}$  iff the word $uv$
of size $2n$ is a legal word in $Y$.}
\end{equation}
 (Observe that this is not the standard
transition matrix for the $n$-word recoding of a subshift of finite type.) 
There is a natural semiconjugacy $Z^{(n)}\to Y$ from the shift of finite type 
$Z^{(n)}$ to $\si^n$ on $Y$, and hence the topological entropy of $Z^{(n)}$ 
exceeds $n \,h_{\text{top}}(Y)$. 
Let $h=\tfrac 1n h_{top}(Z^{(n)})\ge h_{top}(Y)$.
We have that 
\begin{align*}
\#\text{ symbols of $Z^{(n)}$}
= \#\text{ $n$-words in $Y$}
=\#\text{ $n$-cylinders in $Y$}
= K_n\, \ee^{n h_{top}(Y)}
\le K_n \,\ee^{nh},
\end{align*}
where $K_n$ has subexponential growth.
 By Lemma~\ref{Lper} there is a periodic orbit $\Ga_n$
in $Z^{(n)}$ with period at most $1+ K_n \ee^{nh} \ee^{1-nh}= 1+\ee\,K_n$.
Since each symbol in $Z^{(n)}$ corresponds to a word of length $n$ in the original
shift space $\Om=\Si(A)$, the periodic orbit $\Ga_n$ corresponds to a periodic orbit in $\Si(A)$
of period at most 
$$
P(n):=n\, (1+\ee\, K_n),
$$
which has subexponential growth.

We claim that any $n$-word in the projection of $\Ga_n$
 is a legal $n$-word in $Y$.
For this, observe that if the word is a symbol of $Z^{(n)}$ then it is a
legal $n$-word in $Y$ by the definition of $Z^{(n)}$. If the $n$-word is 
inside a concatenation of two symbols of $Z^{(n)}$, the transition 
rule~\eqref{concate} defining $Z^{(n)}$ implies that it is a legal 
$n$-word in $Y$. It follows, using the metric~\eqref{da}, and identifying $\Ga_n$
with its projection to $\Si(A)$, that
\begin{equation}\label{dgay}
c(\Ga_n,Y):=\sup_{\ow\in\Ga_n}d_a(\ow,Y) \le a^{-n}.
\end{equation}

Recall that the return time $\tau$ to $\cup\fT$ and the ceiling function
in $S(\Om,\otau)$ are bounded above by $\a$.
Let $B>0$ be such that 
\begin{equation}\label{blipphi}
\tt_1,\tt_2\in\La, \quad  |t|\le \a 
\quad\then \quad d(\vr_t(\tt_1),\vr_t(\tt_2))< B\, d(\tt_1,\tt_2).
\end{equation}
Let $\De_n:=\vr_{\re}(\Pi(\Ga_n))=\vr_{[0,\a]}(\Pi(\Ga_n))
\subset \La=\Pi(S(\Om,\otau))$ 
be the periodic orbit in $\La$ which corresponds to $\Ga_n$.
Recall that $\Pi$ is Lipschitz \cite[p. 367]{FH} with the metric~\eqref{da}.
Since $\cA(L)=\vr_{[0,\a]}(\Pi(Y))$, from~\eqref{dgay} and~\eqref{blipphi}
we get that 
\begin{equation}\label{deDen}
c(\De_n,\cA(L))=\sup_{\tt\in\De_n}d(\tt,\cA(L))
\le B\, \Lip(\Pi)\, a^{-n}.
\end{equation}
The period $T(\De_n)$ of $\De_n$ is bounded by
\begin{equation}\label{TDen}
T(\De_n)\le \a\, P(n).
\end{equation}
Since $P(n)$ has subexponential growth from 
\eqref{TDen} and \eqref{deDen} we have that
\begin{align*}
\liminf_{T\to+\infty} \;T^\ell \Big[\inf_{\mu\in\cP_L(T)} c(\mu,\cA(L))\Big] 
\le \liminf_n T(\De_n)^\ell \, c(\De_n,\cA(L))
=0.
\end{align*}
\end{proof}

\medskip

\begin{Corollary}\label{cmn}
Suppose that the Aubry set $\cA(L)$ is hyperbolic. 
There is   a sequence
of periodic orbits $\mu_n$ with periods $T_n$ and $m_n\in\na^+$, $m_n>n$,
  such that 
for any $0<\be<1$  and any
$k\in\na^+$
\begin{equation}\label{ecmn}
\int d\big(\tt,\cA(L)\big)\,d\mu_n(\tt) = o(\be^{k\, m_n}) 
\qquad {\rm and} \qquad
\lim_n \frac{\log T_n}{m_n}=0.
\end{equation}
\end{Corollary}

\begin{proof}
Observe that it is enough to prove the Lemma for $k=1$.
By Proposition~\ref{Pper} there is a sequence of periodic orbits $\mu_n$ 
with periods $T_n\to\infty$ such that for any $\ell\in\na^+$
$$
\lim_n T_n^\ell \left(\int d(\tt,\cA(L))\; d\mu_n(\tt)\right) =0.
$$
Let 
$$
r_n:= \log_\be\left( \int d(\tt,\cA(L))\, d\mu_n(\tt)\right).
$$
Since 
$$
\be^{r_n}\le T^\ell_n \,\be^{r_n}\le 1
\qquad \Longleftrightarrow \qquad
-\frac 1{\ell}\le \frac{\log_\be T_n}{r_n} \le 0
$$
we have that $r_n^{-1}\log T_n \to0$.
Define $m_n:=\lfloor \tfrac 12 r_n\rfloor$, then $m_n^{-1}\log T_n \to0$ and
$$
\int d(\tt,\cA(L))\, d\mu_n(\tt) =\be^{r_n}\le \be^{m_n+\frac 12 r_n}=o(\be^{m_n}).
$$
\end{proof}

\medskip

 \begin{Lemma}\label{A.1}
  Let $a_1,\ldots, a_n$ be non-negative real numbers, and let $A=\sum_{i=1}^n a_i\ge 0$.
  Then
  $$
  \sum_{i=1}^n -a_i\log a_i \le 1 + A\,\log n,
  $$
  where we use the convention $0\, \log 0 = 0$.
  Moreover,
    $$
  \text{if  $A=1$ then \qquad}\sum_{i=1}^n -a_i\,\log a_i \le \log n.
  $$
  \end{Lemma}
  
  \begin{proof}
  Applying Jensen's inequality to the concave function $x\mapsto -x \log x$ yields
  $$
  \frac 1n \sum_{i=1}^n-a_i\log a_i 
  \le -\left(\frac 1n \sum_{i=1}^n a_i\right)
  \log\left(\frac 1n \sum_{i=1}^n a_i\right)
  =-\frac An\,\log A + \frac An\, \log n
  $$
  from which the result follows.
  When $A=1$ use that $1\cdot\log 1=0$ in the previous inequality.
 \end{proof}

 \medskip

 Recall that  $\cM_{\text{inv}}(L)$ is the set of Borel probabilities in $TM$ invariant under the 
 Lagrangian flow.

 \medskip                     
 
 \begin{Lemma}\label{LA2}\quad
 
 Let  $L$ be a Tonelli Lagrangian and $e>c(L_0)$.
  There is  $B=B(L,e)>0$ 
 such that for every 
 $\nu\in \cM_{\text{\rm inv}}(L)$ with $\supp(\nu)\subset [\text{\sl Energy}(L)<e]$,
 $$
 \int L\, d\nu \le -c(L) + B(L,e)\int d\big(\tt,\cA(L)\big)\; d\nu(\tt).
 $$
 \end{Lemma}
 
 \begin{proof}
 From Bernard~\cite{Be3} after Fathi and Siconolfi~\cite{FaSi} 
 we know that there is a critical
 subsolution $u$ of the Hamilton-Jacobi equation  for the Hamiltonian $H$ of $L$, i.e.
  \begin{equation}\label{sHJ}
 H(x,d_xu)\le c(L),
 \end{equation}
  which is
 $C^{1}$ with Lipschitz derivatives. Let 
 $$
 \L(x,v):= L(x,v)+c(L)-d_xu(v).
 $$
 Inequality~\eqref{sHJ} implies that $\L\ge 0$. 
 Also  from
   \eqref{aubry}, $\L|_{\cA(L)}\equiv 0$. 
 The Aubry set $\cA(L)$ is included in the 
 energy level $c(L)$ (e.g. Ma\~n\'e~\cite[p. 146]{Ma7}), hence it is compact.
There is a Lipschitz constant $B$ for the function $\L$ on the convex $[E(L)<e]$:
 $$
 \forall\tt\in [E(L)<e]\qquad \L(\tt) \le 0 + B\, d(\tt,\cA(L)) .
 $$ 
 By Birkhoff ergodic theorem every invariant probability is closed:
 \begin{align*}
 \int du(x,v) \;d\nu &
 = \int\left[ \lim_{T\to\infty}\frac 1T\int_0^T du(\vr_t(\tt))\; dt\right] d\nu(\tt)
 \\
 &=\int \lim_{T\to\infty}\left[\frac{u(\pi(\vr_T(\tt)))-u(\pi(\tt))}T\right] d\nu(\tt) =0.
 \end{align*}
 Therefore
 $$
 \int \big(L+c(L)\big) \;d\nu =\int\L\,d\nu \le B \int d(\tt,\cA(L))\;d\nu.
 $$
 \end{proof}

\bigskip

\medskip
\begin{Lemma}\label{partAB}\quad

Let $N$ be a compact riemannian manifold and $\mu$ a 
Borel probability on $N$.

Given $h>0$ there exists a finite Borel partition 
$\A=\{A_1,\ldots, A_r\}$ of $N$ with the following properties:
\begin{enumerate}
\item\label{partAB1} $ \diam\A< h$,
\item\label{partAB2} $\forall A\in\A\quad\mu(\partial A)=0$,
\item\label{partAB3} $\forall\e>0$ \; $\forall A_i\in\A$ \; $\exists B_i\subset A_i$ such that 
$B_i$ is compact, $\mu(\partial B_i)=0$, $\mu(A_i\setminus B_i)<\e$. 
\end{enumerate}

\begin{proof}
We first show that
for any $x\in N$ there is a ball $B(x,r)$, $r<\tfrac 12 h$ such that $\mu(\partial B(x,r))=0$.
Indeed if $h$ is small the sets $F_r:=\partial B(x,r)$, $0<r<h$ are disjoint. Since $\mu$ is
finite, at most a countable number of the sets $F_r$ can have positive measure.   
Let $\cO=\{U_i\}_{i=1}^m$ be a finite cover of $N$ by open balls with $\mu(\partial U_i)=0$, 
$\diam U_i<h$ and such that $U_i\setminus\cup_{j\ne i}\ov{U_j}\ne \emptyset$. 
Define inductively $A_1:=U_1$, $A_{i+1}:=U_{i+1}\setminus\cup_{j\le i}U_j$. Then 
$\ov{A_i}=\ov{\intt A_i}$ and 
$\A:=\{A_i\}_{i=1}^m$ is a Borel partition of $N$ 
satisfying~\eqref{partAB1} and~\eqref{partAB2}.

For $r>0$ small let 
$$
B_i(r):=\{\,x\in A_i\;:\; d(x,A_i^c)\ge r\,\}.
$$
We have that $B_i(r)$ is compact and $B_i(r)\uparrow \intt A_i$. 
Thus $\lim_{r\to 0}\mu(B_i(r))=\mu(\intt A_i)=\mu(A_i)$. Also $\partial B_i(r_1)\cap\partial B_i(r_2)
=\emptyset$ if $r_1\ne r_2$ because
$$
\partial B_i(r)=\{x\in A_i\;:\;d(x,A^c)=r\}.
$$ 
Therefore there is $r_i>0$ such that $B_i:=B_i(r_i)$ satisfies
$\mu(A_i\setminus B_i)<\e$ and $\mu(\partial B_i)=0$.

\end{proof}

\end{Lemma}

\bigskip
\section{Uniform h-expansivity of hyperbolic Ma\~n\'e sets.}
\label{suexpan}
\quad

In order to use Theorem~\ref{ESS} from appendix~\ref{AE} we need to establish
the uniform h-expansivity of our hyperbolic Ma\~n\'e sets.
Recall from Definition~\ref{hexpan} that a homeomorphism $f:X\to X$ is
{\it entropy expansive} or {\it h-expansive} if there is $\e>0$ such that 
\begin{gather}
\forall x\in X \qquad h_{\rm top}(\Ga_\e(x,f), f)=0,  \qquad \text{where}
\notag\\
\Ga_\e(x,f):=\{ y\in X \; |\; \forall n\in \Z \quad d(f^n(x),f^n(y))\le \e\;\}.
\label{Gae}
\end{gather}

Also from Definition~\ref{duhe}, given $\cU\subset C^0(X,X)$ and $Y\subset X$ compact;
we say that $\cU$ is {\it uniformly h-expansive} on $Y$ if there is $\e>0$ such that 
$$
\forall f\in \cU\quad \forall y\in Y\qquad h_{\rm top}(\Ga_\e(y,f),f)=0.
$$

In section~\ref{sssenco} we use $f=\vr_1^{L+\phi}|_{\E_\phi\cap U}$, 
the time 1 map of the lagrangian flow restricted to a neighbourhood $U$ of the
hyperbolic Ma\~n\'e set. No Markov partition is used in section~\ref{sssenco}.
For a hyperbolic flow the set~\eqref{Gae} is an orbit segment of length $2 L\e$
if $\e$ is smaller than its expansivity constant and $L>0$ is a constant,
using the
expansivity definition \cite[Def. 1.7.2]{FH}, the Shadowing Lemma \cite[Thm. 5.3.3]{FH},
and the expansivity of hyperbolic sets \cite[Cor. 5.3.5]{FH}.

The expansivity of a hyperbolic set extends to a neighborhood of the hyperbolic set,
see Corollary~5.3.5 in \cite{FH}.  Also the hyperbolicity of an invariant set $\La$ extends to
the maximal invariant set $\La^U_\phi =\bigcap_{t\in\re}\phi_t(\ov U)$ for $C^1$ perturbations
of the flow $\vr$, see Proposition~5.1.8 in \cite{FH} on persistence of hyperbolicity. 
The $C^1$ continuity of the lagrangian flow restricted to the energy level is
proved in Appendix~\ref{asst} and stated in Remark~\ref{remell}.

Then the upper semicontinuity of  the Ma\~n\'e set \cite[Lemma~5.2]{CP} implies that 
$C^2$ perturbations of the Lagrangian have hyperbolic Ma\~n\'e sets in a neighborhood 
of the original Ma\~n\'e set. They inherit the expansivity. The uniform h-expansivity amounts 
to obtain a uniform expansivity constant $\e$ and a uniform shadowing Lipschitz constant $L$
(cf.~\cite[Def. 5.3.1]{FH}, \cite[Thm. 5.3.3]{FH}). These constants depend on the constants of
the hyperbolicity definition that can be taken uniformly on pertubations 
(see the proof of the persistence of hyperbolicity in Proposition~5.1.8 \cite{FH}).
Or use Remark~\ref{rue} or Corollary~\ref{Rfe} in appendix~\ref{asha}.

The careful reader will find uniform proofs of the shadowing lemmas 
for hyperbolic flows in appendix~\ref{asha} and uniform entropy 
expansivity in Remark~\ref{rue} and Corollary~\ref{Rfe}. 
The continuity in the $C^{1}$ topology 
of the lagrangian vector field in the critical energy level $c(L)$
under perturbations of the potential $\phi$ is proved in appendices
\ref{asst} and \ref{ashms} and finally stated in Remark~\ref{remell}.

\bigskip

\section{Small entropy nearby closed orbits.}\quad
\label{sssenco}

Recall that $\cM_{\text{inv}}(L)$ is the set of  Borel probabilities in $TM$ which 
are invariant under the Lagrangian flow and $\cM_{\min}(L)$ is the set 
of minimizing measures~\eqref{Mmin}.
For $\mu\in\cM_{\text{inv}}(L)$ denote by $h(L,\mu)$ the entropy of $\mu$ under the Lagrangian flow $\vr_t$.

\medskip

\noindent{\bf Proof of Theorem~\ref{Tzeroentropy}:}

For $\ga>0$ write
\begin{align*}
\cH^k(L_0):&=\{\,\phi\in C^k(M,\re)\;|\; \mN(L_0+\phi) \text{ is hyperbolic }\},
\\
\cE_\ga:&=\{\,\phi\in \cH^k(L_0)\;|\; 
 \forall \mu\in\cM_{\min}(L_0+\phi)
\quad  h(L_0+\phi,\mu) < \ga  \;\}.
\end{align*}

The upper semicontinuity of the Ma\~n\'e set (cf. \cite[Lemma~5.2]{CP}) 
and the persistence of hyperbolicity \cite[Prop. 5.1.8]{FH}
imply that the set $\cH^k(L_0)$ is open.
As an open subset $\cH^k(L_0)\subset C^k(M,\re)$ of a complete metric
space we have that $\cH^k(L_0)$ is a Baire space.

It is enough to prove that for every $\ga>0$ the set $\cE_\ga$ is  open and dense in $\cH^k(L_0)$ 
because in that case using  Corollary~\ref{cecmm} and the variational principle for the entropy
 \cite[Cor. 4.3.9]{FH}, 
 we have that
\begin{align*}
\bigcap\limits_{n\in\na^+}\cE_{\frac 1n}
&=\big\{\,\phi\in\cH^k(L_0)\;\big|\;\forall \mu\in
\cM_{\min}(L_0+\phi)
\quad 
h(L_0+\phi,\mu)=0\,\big\}
\\ 
&=\big\{\, \phi\in\cH^k(L_0)\;\big|\; \mN(L_0+\phi) \text{ has zero topological entropy}\,\big\}
\end{align*}
is a residual set.

\medskip

{\it Step 1.} {\sl $\cE_\ga$ is $C^k$ open.}

The map $C^k(M,\re)\ni\phi\mapsto\mN(L_0+\phi)$ is upper semicontinuous 
(cf. \cite[Lemma~5.2]{CP}). See also Remark~\ref{remell} for the 
continuity of the lagrangian vector field restricted to the critical energy level
in the $C^{k-1}$ topology.\footnote{In this argument we only need this for $k=2$.}

Therefore, given $\phi_0\in\cH^k(L_0)$, there are neighbourhoods 
$U$ of $\mN(L_0+\phi_0)$ in $TM$
and $\cU$ of $\phi_0$ in $C^k(M,\re)$ 
such that for any $\phi\in\cU$, $\mN(L_0+\phi)\subset U$ and
$\mN(L_0+\phi)$ is a hyperbolic set for $\vr^{L_0+\phi}_t$ 
restricted to
$\E_\phi:=E^{-1}_{L+\phi}\{c(L+\phi)\}$.
Moreover, 
by the persistence of hyperbolicity \cite[Prop. 5.1.8]{FH},
we can choose $\cU$ and $U$ such that for every $\phi\in\cU$
the lagrangian flow $\vr^{L_0+\phi}_t$ of $L_0+\phi$ is hyperbolic in the maximal invariant
subset of $\ov U$ and by definition \ref{hexpan}, and section~\ref{suexpan},
the set of maps   
$\{\,\vr^{L_0+\phi}_1|_{\E_\phi\cap\ov U}\,:\,\phi\in\cU\,\}$ 
is a
uniformly h-expansive family on $\ov U$.
In particular $\cH^k(L_0)$ is open in $C^k(M,\re)$.

Consider the subset $\cS\subset C^k(M,\re)$ of potentials $\phi$
for which $\mN(L_0+\phi)$
contains a singular point for the Lagrangian flow. By \cite[Theorem~C]{CP}
there is an open and dense subset $\cO_1\subset \cS$ such that 
if $\phi\in\cO_1$ then $\mN(L_0+\phi)$ is a single singularity of $\vr^{L_0+\phi}$.
Then by \cite[Theorem~D]{ham} generically this singularity is hyperbolic,
i.e. $\cO_1\cap\cH^k(L_0)$ is open and dense in $\cO_1$.

So we can restrict our arguments to Ma\~n\'e sets without singularities.
In this case by Corollary~\ref{CSH1} 
we can identify the energy level
$\E_\phi\supset\mN(L_0+\phi)$ with the unit tangent bundle $SM$ by the
radial projection.

Suppose that $\phi_n \in \cH^k(L_0)\setminus \cE_\ga$, $\phi_0\in \cH^k(L_0)$
and $\lim_n\phi_n=\phi_0$ in $C^k(M,\re)$.  Then there are 
$\nu_n\in\cM_{\min}(L_0+\phi_n)$ with $h(L_0+\phi_n,\nu_n)\ge \ga$.
The map $\phi\mapsto c(L_0+\phi)$ is continuous (cf. \cite[Lemma~5.1]{CP})
and $\supp\nu_n$ is in the energy level 
$\E_{\phi_n}=E_{L_0+\phi_n}^{-1}\{c(L_0+\phi_n)\}$
(cf. Carneiro~\cite{Carneiro}). Thus we can assume that all the probabilities
$\nu_n$ are supported on a fixed compact subset $\K$ of $TM$.
 Taking a subsequence if necessary we can assume that 
 $\nu_n\to\nu\in \cM(L_0+\phi_0)$ in the weak* topology.

 The map $(\mu,\phi)\mapsto \int (L_0+\phi)\,d\mu$ is continuous 
 with respect to  $\lV \phi\rV_{\sup}$ and to the weak*
 topology on the set of Borel probabilities on $\K$.
 Also the map $\phi\mapsto c(L_0+\phi)$ is 
 continuous with respect to $\lV \phi\rV_{\sup}$ 
(cf. \cite[Lemma~5.1]{CP}). Using that [see eq.~\eqref{minmeas}]
$$
c(L_0+\phi) = -\min_{\mu\in\cM(L_0+\phi)}\int \bigl(L_0+\phi\big)\,d\mu,
$$
we obtain that the limit $\nu\in\cM_{\min}(L_0+\phi_0)$.

  By Corollary~\ref{CSH1} we can identify the energy levels $\E_{\phi_n}$ with the unit 
tangent bundle $SM$ under the radial projection 
$R(\phi_n):\E_{\phi_n}\to SM$. 
Since $\phi\mapsto c(L_0+\phi)$ is continuous, we have that 
the projected lagrangian vector fields 
$R(\phi_n)_*X(L_0+\phi_n)|_{\E_{\phi_n}}\to R(\phi_0)_*X(L_0+\phi_0)\vert_{\E_{\phi_0}}$
converge in the $C^{k-1}$ topology on $SM$ (see Remark~\ref{remell}).

By Corollary~\ref{cecmm} we have that $\supp \nu_n\subset\ov U$.
By section~\ref{suexpan} or Remark~\ref{rue}, the family of conjugated  lagrangian flows 
$\psi^n_t := R(\phi_n)\circ\vr^{L+\phi_n}_t\circ R(\phi_n)^{-1}$, $n\ge 0$
in $SM$ is 
uniformly h-expansive on their maximal invariant sets of
$R(\ov U)$.
Applying Theorem~\ref{ESS} we get that
$$
\ga\le \limsup_n h({L_0+\phi_n},\nu_n)\le h({L_0+\phi_0},\nu).
$$
Therefore $\cH^k(L_0)\setminus\cE_\ga$ is relatively closed in 
(the open set) $\cH^k(L_0)$
and hence $\cE_\ga$ is open in $C^k(M,\re)$.

\bigskip

\pagebreak

{\it Step 2.} {\sl Density.}

We have to prove that $\cE_\ga$ intersects every non-empty
open subset of $\cH^k(L)$.
Let 
\linebreak
$\cU_1\subset \cH^k(L)$ be open and non-empty. 
 By Ma\~n\'e~\cite[Thm.~C.(a)]{Ma6} there is $\phi_0\in\cU_1$ such that  
 $\#\cM_{\min}(L_0+\phi_0)=1$.
 Write 
 \begin{equation}\label{LL0p0}
 L:=L_0+\phi_0.
 \end{equation}

Let $m_n$, $T_n$, $\mu_n$ be given by Corollary~\ref{cmn} for $L$.
Let $\Ga_n:=\supp\mu_n$ and let $\Ga_n(t)$ be the associated periodic orbit
with its parametrization. Corollary~\ref{cmn} and Lemma~\ref{LA2} are proven
for a single lagrangian. The reader can check that in the following proof they
are only applied to the lagrangian $L$ in~\eqref{LL0p0}.

 For $\phi\in C^k(M,\re)$ near $0$ write
$\E_\phi:=E_{L+\phi}^{-1}\{c(L+\phi)\}$.

\refstepcounter{Thm}
\begin{Claim}\label{c1}\quad

There are $0<\be<1$, 
$N^1_\ga>0$, $C(\cU_1,\be)>0$
 and a neighbourhood 
$0\in\cU_2\subset\cU_1-\phi_0$ such that if 
$n>N^1_\ga$, $\phi\in\cU_2$, 
$\mu\in\cM_{\min}(L+\phi)$,
$\Ga_n(t)$ is also a periodic orbit for $L+\phi$ and
$h(\mu)> 3 \ga \;C(\cU_1,\be)$ then 
\begin{equation*}\label{mu>ga}
\mu(\{\,\tt\in \E_\phi \;|\; d(\tt,\Ga_n)\ge \be^{m_n}\,\})>\ga.
\end{equation*}
\end{Claim}

\noindent{\it Proof  Claim~\ref{c1}:}

 The hyperbolicity of $\mN(L)$ and the upper semicontinuity of 
 the Ma\~n\'e set imply that there is $h>0$ and an open subset
  $0\in\cU_{12}\subset\cU_1-\phi_0$ such that $h$ is a uniform 
 h-expansivity constant (cf. Section~\ref{suexpan}, Remark~\ref{rue}) 
 for all $\vr_1^{L+\phi}|_{\mN(L+\phi)}$, with $\phi\in\cU_{12}$.

  Using Corollary~\ref{CSH1} identify the energy levels $\E_\phi:=E_{L+\phi}^{-1}\{c(L+\phi)\}$ 
  with the unit tangent bundle $SM$ using the radial projection $R(v) = \tfrac{v}{|v|}$.
 Let $\mu_0$ be the minimizing measure for $L$:
  \begin{equation}\label{uniqmu0}
  \cM_{\min}(L)=\{\mu_0\}.
  \end{equation}
  Let $\A=\{A_1,\ldots,A_r\}$ be a finite Borel partition of the energy level
  $\E_\phi$
  with 
  \begin{equation}\label{diamah}
  \diam\A<h
  \end{equation}
  and $\mu_0(\partial A)=0$ for all $A\in\A$.
  Let 
  \begin{equation}\label{defer}
  0<\e<  \tfrac 12 (r\,\log r)^{-1},
  \end{equation}
  where $r=\#\A$.
  Using Lemma~\ref{partAB}, 
 for each $A_i\in \A$ let $B_i\subset A_i$ be a compact set such that 
 $\mu_0(A_i\setminus B_i)<\e$ and $\mu_0(\partial B_i)=0$. 
 Define the partition $\B:=\{B_0, B_1,\cdots,B_r\}$,
 where $B_0:=\E_\phi\setminus \bigcup_{i=1}^r B_i$.

  By the continuity of the critical value map $\phi\mapsto c(L+\phi)$ and 
  the ergodic characterization~\eqref{minmeas} we have that any weak* limit
  of minimizing measures is a minimizing measure.
  By the uniqueness~\eqref{uniqmu0} and the compactness of $\E_\phi$,
  if $\lim_k\phi_k= 0$ and $\nu_k\in\cM_{\min}(L+\phi_k)$ then 
  $\lim_k\nu_k=\mu_0$. Since $\mu_0(\partial A_i)=0=\mu_0(\partial B_i)$
  and $\partial(A_i\setminus B_i)\subset \partial A_i\cup \partial B_i$ 
  we have that
  $$
  \lim_k\phi_k=0,\; \nu_k\in\cM_{\min}(L+\phi_k) 
  \qquad \then \qquad\forall i \quad 
   \lim_k\nu_k(A_i\setminus B_i) =\mu_0(A_i\setminus B_i).
  $$
  Then there is an open set $0\in\cU_{13}\subset \cU_{12}$ such that 
  \begin{equation}\label{cU13}
  \forall \phi\in\cU_{13} \quad\forall \mu\in\cM_{\min}(L+\phi)\quad
  \forall i\le r\qquad
   \mu(A_i\setminus B_i)<\e.
  \end{equation}

  Let $\bO:=\{B_0\cup B_1,\ldots, B_0\cup B_r\}$. Observe that $\bO$ 
is an open cover because 
\linebreak
$B_0\cup B_i=(\cup_{j\ne i} B_i)^c$.
\begin{equation}\label{defde}
\text{
Let $\de>0$ be a Lebesgue number for $\bO$.
}
\end{equation}
Let $\be>0$ be such that 
\begin{equation}\label{defbe1}
0<\be<\min\Big\{\,\frac \de 2,\;
\inf_{\phi\in\cU_1}\Lip(\vr^{L+\phi}_1|_{\E_\phi})^{-1}\Big\}
\end{equation}
and such that 
\begin{equation}\label{defbe2}
\sup_{\phi\in\cU_1}\sup_{|\tau|\le\be}
\sup_{\tt\in\E_\phi}d\big(\vr^{L+\phi}_\tau(\tt),\tt\big)
<\frac \de 2.
\end{equation}

Given $n\in\na$ and $\a>0$, let $G(n,\a)$ be a cover of 
$\E_\phi$ of minimal cardinality  by balls of radius $\a^n$. Let
\begin{equation}\label{coverG}
C(\cU_1,\a):=\limsup_n \tfrac 1n\log\# G(n,\a).
\end{equation}
Then (cf. Falconer~\cite[Prop.~3.2]{Falconer0})
$$
1\le \lim_{\a\to 0} C(\cU_1,\a) = \dim \E_\phi<\infty.
$$
Shrink $\be$ if necessary so that $\be$ satisfies \eqref{defbe1}, \eqref{defbe2} and
\begin{equation*}\label{defbe3}
\tfrac 12\le C(\cU_1,\be) < \infty.
\end{equation*}

Let $Q\in\na^+$ be such that 
\begin{equation}\label{q2l2}
\frac{3+2\log 2}Q < \tfrac 14\,\ga\, C(\cU_1,\be).
\end{equation}
Let 
\begin{equation}\label{defnga}
N_\ga^1> 3Q>0
\end{equation}
 be such that 
\begin{align}
\forall n>N_\ga^1\qquad
\tfrac 1n \log\#G(n,\be) &\le 2 \, C(\cU_1,\be)
\label{ngagn}
\\
\intertext{and, using \eqref{ecmn},}
\forall n> N_\ga^1\qquad
\tfrac1{m_n}\log T_n -\tfrac 1{m_n}\log\be 
&< 
\tfrac 14\,\ga\,C(\cU_1,\be).
\label{ngatn}
\end{align}
Suppose that  $\phi\in\cU_{13}$, 
$\mu\in\cM_{\min}(L+\phi)$, $\ga>0$ and $n>N_\ga^1$ satisfy
\begin{equation}\label{mulega1}
\mu(\{\,\tt\in \E_\phi \;|\; d(\tt,\Ga_n) \ge \be^{m_n}\,\})\le\ga,
\end{equation}
and that $\Ga_n$ is also a periodic orbit for $\vr^{L+\phi}$,
we shall prove that then 
\begin{equation}\label{hmulphi3gac}
h_\mu(\vr^{L+\phi})\le3\,\ga\, C(\cU_1,\be).
\end{equation}
This  implies Claim~\ref{c1}.

 Observe that 
 $$
 A_i\cap B_j=
 \begin{cases}
 \emptyset &\text{if }\quad  i\ne j\ne 0,\\
 B_j &\text{if } \quad i=j\ne 0,\\
 A_i\setminus B_i &\text{if }\quad j=0.
 \end{cases}
 $$
 Let $\rho(x):=-x\log x$, $x\in[0,1]$, with $\rho(0):=0$. Then
 \begin{equation}\label{rhoab}
 \rho\left(\frac{\mu(A_i\cap B_j)}{\mu(B_j)}\right)
 =\begin{cases}
 0 &\text{ if }\quad j\ne 0, \\
 \rho\left(\frac{\mu(A_i\setminus B_i)}{\mu(B_0)}\right)
 &\text{ if }\quad j=0.
 \end{cases}
 \end{equation}
 Observe that $B_0=\cup_{i=1}^r(A_i\setminus B_i)$, then
 from~\eqref{cU13}, 
 \begin{equation}\label{mub0e}
 \mu(B_0)< r\,\e.
 \end{equation}

 We have that the relative entropy satisfies
 \begin{align*}
 H_\mu(\A|\B):&=-\sum_{i=1}^r\sum_{j=0}^r \mu(A_i\cap B_j)
  \;\log\frac{\mu(A_i\cap B_j)}{\mu(B_j)}
  \\
  &= \sum_{i=1}^r\sum_{j=0}^r \mu(B_j) \,
  \rho\left(\frac{\mu(A_i\cap B_j)}{\mu(B_j)}\right)
  \\
  &=\sum_{i=1}^r \mu(B_0)\;\rho\left(\frac{\mu(A_i\cap B_0)}{\mu(B_0)}\right)
  \qquad \text{ using \eqref{rhoab},}
  \\
  &\le \mu(B_0) \log r \qquad \text{by Lemma~\ref{A.1} with A=1},
  \\
  &\le \e\, r\log r  < 1     \qquad\text{ using \eqref{mub0e} and \eqref{defer}.}
 \end{align*}
  From Walters \cite[Theorem~4.12(iv)]{Walters} for all 
  $f:\E_\phi\to\E_\phi$ continuous and $Q\in\na$ we have that
  \begin{align}\label{h+1}
  h_\mu(f^Q,\A)\le h_\mu(f^Q,\B)+H_\mu(\A|\B)
  \le  h_\mu(f^Q,\B) +1.
  \end{align}

 Define
 $$
 B(\Ga_n,\be^{m_n}):=\{\,\tt\in\E_\phi\;|\;d(\tt,\Ga_n)<\be^{m_n}\,\}.
 $$
 Let $f$ be the time 1 map  $f:=\vr^{L+\phi}_1$.
 Fix $\tt_0\in\Ga_n$ and let
 $$
 R_n:=\big\{\,\vr_{i\be}(\tt_0)\;\big|\; i=0, 1,\ldots,\lfloor\tfrac{T_n}\be\rfloor\,\big\}.
 $$
 We claim that $R_n$ is an $(m_n,\de,f)$-generating set for $B(\Ga_n,\be^{m_n})$,
 i.e.
 \begin{equation}\label{rngenerating}
 B(\Ga_n,\be^{m_n})\subset \bigcup_{p\in R_n}V(p,m_n,\de,f),
 \end{equation}
 where
 $$
 V(p,m_n,\de,f):=\{\,\tt\in\E_\phi\;|\; d(f^i(\tt),f^i(p))<\de, 
 \quad\forall i=0,\ldots, m_n-1\,\}.
 $$
 Indeed if $q\in B(\Ga_n,\be^{m_n})$ there is $\tt\in\Ga_n$ such that 
 $d(q,\tt)<\be^{m_n}$. 
 Recall that by hypothesis $\Ga_n(t)$ is a periodic orbit
 of $\vr^{L+\phi}$. There is $p\in R_n$ such that $p=\vr^{L+\phi}_\tau(\tt)$
 with $|\tau|\le \be$. 
  In particular
 $$
 f^j(p)=\vr^{L+\phi}_j(\vr_\tau^{L+\phi}(\tt))=\vr_\tau^{L+\phi}(f^j(\tt)).
 $$
 If $0\le j\le m_n-1$, by~\eqref{defbe1} and~\eqref{defbe2}
  we have that 
\begin{align*}
d\big(f^j(q),f^j(\tt)\big) 
&\le \Lip(f)^j\, d(q,\tt)\le \Lip(f)^j\,\be^{m_n}\le \be <\de/2,
\\
d\big(f^j(\tt),f^j(p)\big)
&= d\big(f^j(\tt),\vr_\tau^{L+\phi}(f^j(\tt))\big)
\le d_{C^0}(\vr^{L+\phi}_\tau,id)< \de/2,
\\
d\big(f^j(q),f^j(p)\big)
&\le d\big(f^j(q),f^j(\tt)\big) +d\big(f^j(\tt),f^j(p)\big)
< \de/2+\de/2=\de.
\end{align*}
Therefore $q\in V(p,m_n,\de,f)$ and $p\in R_n$.

  Recall that $Q\in\na^+$ is from~\eqref{q2l2}. Write
 \begin{align*}
 \B_{f^Q}^{(k)}:&=\textstyle\bigvee\limits_{i=0}^{k-1} f^{-iQ}(\B),
 \\
 \bO_{f^Q}^{(k)}:
 &=\big\{\textstyle\bigcap_{i=0}^{k-1}f^{-iQ}(U_i)\;
 \big|\; U_i\in \bO \quad \forall i=0,\ldots,k-1\,\big\}.
 \end{align*}
 Let 
 \begin{equation}\label{defW}
 W(n,k,Q):=\big\{\, \cB\in \B^{(k)}_{f^Q}
 \;\big |\; \cB\cap B(\Ga_n,\be^{m_n})\ne\emptyset\,\big\}.
 \end{equation}
 Let $k_n=k(n,Q)\in\na$ be such that
 \begin{equation}\label{kn1}
 (k_n-1) \, Q < m_n \le k_n \,Q.
 \end{equation}
 Since in~\eqref{ecmn} $m_n>n$, 
 by~\eqref{defnga} we have that $k_n\ge 2$ whenever $n>N_\ga^1$.
 We want to estimate $\#W(n,k_n,Q)$.
 Given $\cB\in W(n,k_n,Q)$, choose $q(\cB)\in \cB\cap B(\Ga_n,\be^{m_n})$.
 By the inclusion~\eqref{rngenerating} there is $p(\cB)\in R_n$ such that 
 $$
 q(\cB)\in V(p(\cB),m_n,\de,f)\subset V(p(\cB),k_n,\de,f^Q).
 $$
 In particular $d\big(f^{Qi}(q(\cB)),f^{Qi}(p(\cB))\big)<\de$ for $0\le i<k_n$.
 The choice of $\de$ in \eqref{defde}
  implies that there is $j_i=j_i(p)\in\{1,\ldots,r\}$, depending only on $(i,p)$,
   such that 
 $$
 f^{Qi}(q(\cB))\in B\big(f^{Qi}(p(\cB)),\de\big)
 \subset B_0\cup B_{j_i},
 \qquad i=0,\ldots,k_n-1.
 $$
 Therefore 
 \begin{align*}
 q(\cB)&\in \bigcap_{i=0}^{k_n-1}f^{-Qi}B\big(f^{Qi}(p(\cB)),\de\big)\subset
 \\
 &\qquad\subset \bigcap_{i=0}^{k_n-1}f^{-Qi}(B_0\cup B_{j_i})
 =\bigcap_{i=0}^{k_n-1}\Big(f^{-Qi}(B_0)\cup f^{-Qi}(B_{j_i})\Big).
 \end{align*}
 Since $\B$ is a partition, $\cB\in \B^{(k_n)}_{f^Q}$ 
 and $q(\cB)\in\cB$
 we have that 
 \begin{equation}\label{cbfq}
 \cB=\bigcap_{i=0}^{k_n-1} f^{-Qi}(B_{\ell_i}),
 \qquad \text{ where } \quad \ell_i\in\{0,j_i(p)\}.
 \end{equation}
 Observe that $j_i(p)$ depends only on $i$ and $p$.
Thus for $p\in R_n$  we have that 
 $$
 \#\big\{\,\cB\in W(n,k_n,Q)\subset \B^{(k_n)}_{f^Q}
 \;\big|\; p(\cB) =p\big\}\le 2^{k_n}.
 $$
 Therefore
 \begin{equation}\label{Wnknq}
  \# W(n,k_n,Q)\le 2^{k_n} \cdot \# R_n\le 2^{k_n}\frac{T_n}\be.
 \end{equation}

 By~\eqref{defbe1} we have that 
$B(\tt,\be^n)\subset V(\tt,n,\de,f)$. Therefore
\begin{equation}\label{BVVt}
\forall\tt\in\E_\phi\qquad 
B(\tt,\be^{m_n})\subset V(\tt,m_n,\de,f)\subset V(\tt,k_n,\de,f^Q).
\end{equation}
Identify the covering $G(m_n,\be)$ from~\eqref{coverG}
by balls of radius $\be^{m_n}$
with the set of their centers.
Then  by~\eqref{BVVt}, the set $G(m_n,\be)$ is a $(k_n,\de,f^Q)$-generating set, i.e.
$$
\E_\phi\subset \bigcup_{\tt\in G(m_n,\be)}V(\tt,k_n,\de,f^Q).
$$

We show that the same argument as in~\eqref{Wnknq} gives
for $n>N^1_\ga$ that
 \begin{equation}\label{B2kG}
  \#\B^{(k_n)}_{f^Q}\le 2^{k_n} \,\#G(m_n,\be).
 \end{equation}
 Namely, we construct a map $p: \B^{(k_n)}_{f^Q}\to G(m_n,\be)$
 which is at most $2^{k_n}$ to 1 as follows.
 Given $\cB\in  \B^{(k_n)}_{f^Q}$ choose $q(\cB)\in \cB$.
 Now choose $p(\cB)\in G(m_n,\be)$ such that 
 $$
 q(\cB)\in V(p(\cB),k_n,\de,f^Q).
 $$
 From~\eqref{defde} for all $i=0,\ldots, k_n-1$
  there is $j_i(p)$ such that $B\big(f^{iQ}(p(\cB)),\de\big)\subset B_0\cup B_{j_i(p)}$
  and therefore
  $$
  q(\cB)\in
   \bigcap\limits_{i=0}^{k_n-1}
   f^{-iQ}\big(B(f^{iQ} (p(\cB)),\de)\big)
   \subset\bigcap\limits_{i=0}^{k_n-1}
   f^{-iQ}(B_0\cup B_{j_i(p)}).
  $$
  This implies that 
  $$
  \cB=\bigcap\limits_{i=0}^{k_n-1}
   f^{-iQ}(B_{\ell_i})
   \quad\text{ with }\quad
   \ell_i\in\{0,\,j_i(p(\cB))\},
  $$
  which in turn implies~\eqref{B2kG}.
 
 By  the hypothesis~\eqref{mulega1} and by~\eqref{defW}
 we have that
 $$
 \tga_n:=\sum_{\cB\in\B^{(k_n)}_{fQ}\setminus W(n,k_n,Q)}\mu(\cB)\le \ga.
 $$
 Using Lemma~\ref{A.1} and inequalities
 \eqref{Wnknq} and \eqref{B2kG} we have that
 \begin{align*}
 h_\mu(f^Q,\B)
 &\le\frac 1{k_n}H_\mu(\B^{k_n}_{f^Q})
 \qquad\quad \text{ by~\eqref{HfNdec}  (cf.  Walters~\cite[Theorem~4.10]{Walters})},
 \\
 &\le \frac 1{k_n}\sum_{\cB\in W(n,k_n,Q)}\hskip -.5cm-\mu(\cB)\,\log\mu(\cB)
 +  \frac 1{k_n}\sum_{\cB\in \B^{(k_n)}_{f^Q}\setminus W(n,k_n,Q)}
 \hskip-.9cm -\mu(\cB)\,\log\mu(\cB)
 \\
 &\le \frac 1{k_n}\big(1+(1-\tga_n) \log\# W(n,k_n,Q)\big)
 +\frac 1{k_n} \big( 1 +\ga \, \log\#\B^{(k_n)}_{f^Q}\big)
 \\
 &\le(1+\ga) \log 2 + \tfrac 1{k_n}(2+\log T_n-\log\be)
+\tfrac 1{k_n} \,\ga\, \log\# G(m_n,\be).
 \end{align*}
 Thus, using~\eqref{h+1}, and then~\eqref{ngagn} and that $m_n>n$
 we have that for $n>N_\ga^1$,
 \begin{align}
 h_\mu(f^Q,\A) &\le 1+h_\mu(f^Q,\B) 
 \notag \\
 &\le 1+2\log 2+\tfrac 1{k_n}(2+\log T_n-\log\be)
 +\tfrac {m_n}{k_n}\, 2\,\ga\, C(\cU_1,\be).
 \label{hfqa}
 \end{align}
 The hyperbolicity of $\mN(L+\phi)$ implies that
  $f^Q$ has also h-expansivity constant $h$ on $\mN(L+\phi)$ 
  because for $\tt\in\mN(L+\phi)$ there is $\tau=\tau(\tt)>0$ such that
  $$
  \Ga_h(\tt,f)=\Ga_h(\tt,f^Q)=\vr_{[-\tau,\tau]}(\tt),
  $$
  where
 $$
 \Ga_h(\tt,f^Q):=\{\vrt\in \E_\phi\;|\;\forall n\in\Z\quad d(f^{nQ}(\tt),f^{nQ}(\vrt))<h\}.
 $$
 In particular $h_{top}(\Ga_h(\tt,f^Q))=h_{top}(\Ga_h(\tt,f))=0$.
 By Corollary~\ref{cecmm},
  $\supp(\mu)\subset\mN(L+\phi)$. Therefore $h$ is an h-expansivity 
  constant on $\supp(\mu)$. Since by~\eqref{diamah}
  $\diam \A<h$, by
 Theorem~\ref{TBhE} (cf. Bowen~{\cite[Theorem~3.5]{Bowen9}})
 we have that
 \begin{align*}
 h_\mu(f^Q)= h_\mu(f^Q,\A).
 \end{align*}
 By~\eqref{kn1}, $k_n\ge\tfrac{m_n}Q$,
  \begin{align*}
 h_\mu(f)&=\tfrac 1Qh_\mu(f^Q)=\tfrac 1Q h_\mu(f^Q,\A)
 \\
 &\le \tfrac1Q(3+2\log 2) +\tfrac 1{m_n}{\log T_n}-\tfrac 1{m_n}\log\be
 +2\,\ga\,C(\cU_1,\be).
 \end{align*}
 Using~\eqref{q2l2} and \eqref{ngatn} we obtain
 $$
 h_\mu(f)\le 3\,\ga\, C(\cU_1,\be).
 $$
 This proves inequality~\eqref{hmulphi3gac}, and also
 Claim~\ref{c1}.

 \hfill $\triangle$

\bigskip

\begin{Claim}\label{c2}\quad

There are $C>0$, $N^2_\ga>N^1_\ga>0$ 
such that if $n>N^2_\ga$,
$\phi\in\cU_2$ are such that
$$
\phi \ge 0 \qquad\text{ and } \qquad \phi|_{\pi\Ga_n}\equiv 0,
$$
and $\nu\in\cM_{\min}(L+\phi)$,
$\tt\in \Ga_n$, $\vrt\in\supp(\nu)$,
$d(\tt,\vrt) \ge\be^{m_n}$, then
$$
d\big(\pi(\tt),\pi(\vrt)\big) \ge \tfrac 1C\, \be^{m_n}.
$$
\end{Claim}

Denote 
the action of a $C^1$ curve $\a:[S,T]\to\re$ by
$$
A_L(\a):=\int_S^T L(\a(t),\dot\a(t))\; dt.
$$
The following Crossing Lemma is extracted for Mather~\cite{Mat5} with 
the observation that the estimates can be taken uniformly on $\cU_2$.

\begin{Lemma}[Mather~{\cite[p. 186]{Mat5}}]\label{CL}\quad

If $K>0$, then there exist $\e$, $\de$, $\eta>0$ and 
\begin{equation}\label{CL>1}
C> 1,
\end{equation}
such that 
if $\phi\in\cU_2$, and 
$\a, \ga:[t_0-\e,t_0+\e]\to M$
are solutions of the Euler-Lagrange equation
with $\lV d\a(t_0)\rV$, $\lV d\ga(t_0)\rV\le  K$,
$d\big(\a(t_0),\ga(t_0)\big)\le\de$, and
$$d\big(d\a(t_0),d\ga(t_0)\big)\ge C\; d\big(\a(t_0),\ga(t_0)\big),$$
then there exist $C^1$ curves $a,c:[t_0-\e,t_0+\e]\to M$
such that $a(t_0-\e)=\a(t_0-\e)$, $a(t_0+\e)=\ga(t_0+\e)$,
$c(t_0-\e)=\ga(t_0-\e)$, $c(t_0+\e)=\a(t_0+\e)$, and
\begin{equation}\label{ecl}
A_{L+\phi}(\a)+A_{L+\phi}(\ga)
-A_{L+\phi}(a)-A_{L+\phi}(c)
\ge \eta\; d\big(d\a(t_0),d\ga(t_0)\big)^2.
\end{equation}

\end{Lemma}

\bigskip

\noindent{\it Proof of Claim~\ref{c2}:}
Let $K>0$ be such that
$$
\forall \phi\in\cU_2\qquad
E_{L+\phi}^{-1}\{c(L+\phi)\} \subset \{\,\xi\in TM\;|\; \lV \xi\rV<K\,\}.
$$
Let $\e$, $\de$, $\eta$, $C>0$ be from Lemma~\ref{CL}.
Choose $M_\ga>N^1_\ga$ such that if $n>M_\ga$ then
\begin{equation}\label{1Cd}
\tfrac 1C\be^{m_n} <\de.
\end{equation}
Suppose by contradiction that there are $\phi\in\cU_2$, 
$\nu\in\cM_{\min}(L+\phi)$,
$\tt\in\Ga_n$, $\vrt\in\supp(\nu)$
such that 
$\phi\ge 0$, 
$\phi|_{\Ga_n}\equiv 0$,
 $$
 d(\tt,\vrt)\ge \be^{m_n}
 \qquad\text{ and }\qquad
 d\big(\pi(\tt),\pi(\vrt)\big) < \tfrac 1C\, \be^{m_n}.
$$
Since $\phi$ is $C^2$, $\phi\ge 0$ and $\phi|_{\pi\Ga_n}\equiv 0$
we have that $d\phi|_{\pi\Ga_n}\equiv 0$ and then $\Ga_n(t)$
is also a periodic orbit for $L+\phi$, with the same parametrization.
Write  $\a(t)=\pi\vr^{L+\phi}_t(\vrt)$, $\ga(t)=\pi\vr^{L+\phi}_t(\tt)$, $t\in[-\e,\e]$.
By~\eqref{1Cd} we can apply 
Lemma~\ref{CL} 
and obtain $a,\,c:[-\e,\e]\to M$ satisfying \eqref{ecl}.
Since $\vrt\in\supp(\nu)\subset\cA(L+\phi)$ the segment $\a$ is semi-static
for $L+\phi$:
\begin{align}
&\hskip -.7cmA_{L+\phi+c(L+\phi)}(\a) = \Phi^{L+\phi}_{c(L+\phi)}(\a(-\e),\a(\e))
\notag \\
&=\inf\{\,A_{L+\phi+c(L+\phi)}(x)\;|\;x\in C^1([0,T], M),\;
 T>0, \;
x(0)=\a(-\e), \; x(T)=\a(\e)\;\}.
\label{1semi} 
\end{align}
Consider the curve 
$x = a *\pi \vr^{L+\phi}_{[\e,T_n-\e]}(\tt) *c$ joining $\a(-\e)$ to $\a(\e)$.
Writing
$$
L_\phi:=L+\phi,
$$
we have that
\begin{align}
A_{L_\phi+c(L_\phi)}(x) 
&= A_{L_\phi+c(L_\phi)}(a) + A_{L_\phi+c(L_\phi)}(\Ga_n)-A_{L_\phi+c(L_\phi)}(\ga)
+A_{L_\phi+c(L_\phi)}(c)
\notag\\
& \le A_{L_\phi+c(L_\phi)}(\a) -\eta\, d(\tt,\vrt)^2 + A_{L_\phi+c(L_\phi)}(\Ga_n)
\notag\\
&\le A_{L_\phi+c(L_\phi)}(\a) -\eta\, \be^{2 \, m_n} + A_{L_\phi+c(L_\phi)}(\Ga_n).
\label{Ax}
\end{align}

Now we estimate $A_{L_\phi+c(L_\phi)}(\Ga_n)$. 
Observe that since $L_\phi=L+\phi \ge L$ we have that
$$
c(L+\phi)\le c(L).
$$
Since $\phi\vert_{\Ga_n}\equiv 0$, we have that
$\phi |_{\Ga_n}+ c(L+\phi) \vert_{\Ga_n} \le c(L)$.
Observe that $\mu_n$ is an invariant measure for the
flows of both $L$ and $L+\phi$ with $\supp\mu_n=\Ga_n$,
thus 
\begin{align}
A_{L_\phi+c(L_\phi)}(\Ga_n)
&= T_n\int \big[ L+\phi+c(L+\phi)\big]\; d\mu_n
\notag\\
&\le T_n\int \big[L+c(L)\big]\, d\mu_n
=A_{L+c(L)}(\Ga_n).
\label{LphiL}
\end{align}
By the choice of $m_n$, $T_n$, $\mu_n$ from Corollary~\ref{cmn} 
for $L$ 
and using Lemma~\ref{LA2}  applied to $L$, 
we have that for any $\rho>0$, setting $e:=c(L)+1$,
\begin{align}
A_{L+c(L)}(\Ga_n) &= T_n \int \big[L+c(L)\big]\, d\mu_n
\le B(L,e)\; T_n \int d(\xi,\cA(L)) \,d\mu_n(\xi)
\notag\\
&\le  B\,T_n\cdot o(\be^{k\, m_n}) 
=o\big(\be^{(k-\rho)\,m_n}\big).
\label{AGa}
\end{align}
In Corollary~\ref{cmn} the estimate holds with the same sequence $\mu_n$ for any $k\in\na^+$.
Thus, choosing $k$, $\rho$ such that  $k-\rho\ge2$, if $n$ is large enough, 
from \eqref{Ax}, \eqref{LphiL} and \eqref{AGa}, we get
\begin{align*}
A_{L_\phi+c(L_\phi)}(x) &\le
A_{L_\phi+c(L_\phi)}(\a) -\eta\, \be^{2 m_n} + o(\be^{2 m_n})
\\
&< A_{L_\phi+c(L_\phi)}(\a).
\end{align*}
This contradicts \eqref{1semi} and  proves Claim \ref{c2}. 

\hfill $\triangle$

\bigskip

We use $L$ from~\eqref{LL0p0} and 
$m_n$, $T_n$, $\mu_n$, $\Ga_n$, $\be$ and $\cU_2$ from Claim~\ref{c1} 
and $C$, $N^2_\ga$ from Claim~\ref{c2}.
By Whitney Extension Theorem~\cite[p. 176 ch. VI \S 2.3]{stein} there is $A>0$ 
and $C^k$ functions  $f_n\in C^k(M,[0,1])$ such that
\begin{equation*}
0\le f_n(x) = \begin{cases}
0 & \text{in a small neighbourhood of }\pi(\Ga_n),
\\
\be^{(k+1) m_n} &\text{if } d(x,\pi(\Ga_n)) \ge\tfrac 1C\, \be^{m_n},
\end{cases}
\end{equation*} 
and $\lV f_n\rV_{C^k}\le A$.
Take $\e>0$ such that $\forall n>N^2_\ga$, $\phi_n:=\e f_n\in \cU_2$.

Write $L_n:=L+\phi_n$. 
Observe that $\Ga_n$ is also a periodic orbit for $L_n$.
In particular Claim~\ref{c1} and Claim~\ref{c2} hold for measures
in $\cM_{\min}(L_n)$.
Suppose by contradiction
 that there is a sequence $n=n_i\to+\infty$ 
 such that 
 $\phi_n\notin\cE_{\ov\ga}$ with
$\ov{\ga}= 4\ga\, C(\cU_1,\be)$.
Then there   is a minimizing measure $\nu_n\in \cM_{\min}(L_n)$ with
$h(\nu_n)\ge 4 \ga \,C(\cU_1,\be)>3\ga\, C(\cU_1,\be)$.
By Claim~\ref{c1} and Claim~\ref{c2} we have that
$$
\nu_n\big(\big\{\,\vrt\in TM \;\big|\; d(\pi(\vrt),
 \pi(\Ga_n))\ge\tfrac 1C\,\be^{m_n}\;\big\}\big)>\ga.
$$
Since $\nu_n\in \cM(L_n)\subset \cC(TM)$ is a  closed measure, 
by \eqref{minmeas}, 
$
\int \big[L +c(L)\big]\, d\nu_n \ge 0.
$
Then
\begin{align}
\int \big( L+\phi_n\big) \,d\nu_n &\ge -c(L) +\int\phi_n\,d\nu_n
\notag
\\
&\ge -c(L) + \e \ga\, \be^{(k+1) m_n}.
\label{AM}
\end{align}

Observe that $\mu_n$ is also an invariant probability for $L_n=L+\phi_n$.
From Lemma~\ref{LA2} and Corollary~\ref{cmn}
applied to $L$ and $e=c(L)+1$,  we have that
\begin{align}
\int \big(L+\phi_n\big) \,d\mu_n &
= \int L\;d\mu_n
\le -c(L) + B(L,e) \int d(\theta,\cA(L)) \; d\mu_n(\theta)
\notag
\\
&\le -c(L)+ o\big(\be^{(k+1)m_n} \big).
\label{AN}
\end{align}
Inequalities \eqref{AN} and \eqref{AM} imply
that for $n=n_i$ large enough $\nu_{n_i}$ is not minimizing, 
contradicting the choice of $\nu_{n_i}$.
Therefore $\phi_n\in \cE_{\ov\ga}\cap\cU_1$ for $n$ large enough.

\qed

\bigskip

\appendix

\section{Entropy.}\label{AE}

 Let $\phi_t$ be a continuous flow without fixed points 
 on a compact metric space $X$.
 
 For $T>0$ define the distance $d_T$ on $X$ by
 $$
 d_T(x,y) := \max_{s\in[0,T]} d(\phi_s(x),\phi_s(y)). 
 $$
 For $\de, \, T>0$, $x\in X$ the {\it dynamic ball } $V(x,T,\de)$ is defined 
 as the closed ball of radius $\de$ for the distance $d_T$ centered at $x$, equivalently
 $$
 V(x,T,\de):=\{\,y\in X\;|\; \forall s\in[0,T]\;\; d(\phi_s(x),\phi_s(y))\le \de\;\}.
 $$

 Given $E,\,F\subset X$ we say that $E$ is a $(T,\de)${\it -spanning set} for $F$
 (or that it $(T,\de)${\it -spans} $F$)
 iff
$$
F \subset \bigcup_{e\in E} V(e,T,\de).
$$
 Let 
 $$
 r(F,T,\de):=\min \{\#E \;|\; E\text{ $(T,\de)$-spans } F \,\}.
 $$
 If $F$ is compact, the continuity of $\phi_t$ implies that $r(F,T,\de)<\infty$.
 Let
 $$
 \ov r_\phi(F,\de) := \limsup_{T\to\infty} \frac 1T \log r(F,T,\de).
 $$
 
 We say that $E\subset X$ is $(T,\de)${\it -separated } if
 $$
 x,\,y\in E, \quad x\ne y \quad \then\quad 
 V(x,T,\de)\cap V(y,T,\de) =\emptyset.
 $$
 Given $F\subset X$ let 
 \begin{align*}
 s(F,T,\de):= \max\{\,\#E\,|\, E\subset F,\; E\text{ is $(T,\de)$-separated }\}.
 \end{align*}
 If $F$ is compact,
 Theorem 6.4 in \cite{Walters2}
 shows that $s(F,T,\de)<\infty$.
 Let
 $$
 \ov s_\phi(F,\de):= \limsup_{T\to+\infty}
 \frac 1T \log s(F,T,\de).
 $$
 Define the {\it topological entropy} by
 $$
 h_{top}(F,\phi):=\lim_{\de\to 0}\ov r_\phi(F,\de)
 =\lim_{\de\to 0}\ov s_\phi(F,\de).
 $$
 By Lemma~1 in \cite{Bowen10} these limits exist and are equal.
 
 In fact the topological entropy of a flow $\phi$ equals the topological
 entropy of the homeomorphism $\phi_1$, and more generally
 $h(\phi_t)= |t|\, h(\phi)$, see Proposition~21 in \cite{Bowen10}.
 
  Denote by $\mathfrak B(X)$ the Borel $\si$-algebra of $X$.
  Let $f:X\to X$ be a measurable map.
  Denote by $\cM(f)$ the set of $f$-invariant Borel probabilities on $X$.
 Given a finite Borel partition  $\A\subset\mathfrak B(X)$ 
 of $X$ and an $f$-invariant Borel
  probability $\mu\in\cM(f)$ define its {\it entropy}
  by
  \begin{align*}
  H_\mu(\A):&=-\sum_{A\in\A}\mu(A)\,\log\mu(A).
  \end{align*}
  Given two finite Borel partitions $\A,\,\B$ of $X$ define $\A\vee \B$ as
  \footnote{More generally $\A\vee \B=\si(\A\cup \B)$ is the $\si$-algebra generated by $\A\cup \B$. 
  This is the definition used for an infinite refinement $\bigvee_{i\in\na} \A_i$. } 
  $$
  \A\vee \B:=\{\, A \cap B\;|\; A\in\A,\, B\in\B\,\}.
  $$
  From Walters~\cite[Theorem~4.10]{Walters}, the map
  \begin{equation}\label{HfNdec}
  \na^+\ni N \mapsto \tfrac 1N \, 
  H_\mu\Big(\textstyle\bigvee_{n=0}^{N-1}f^{-n}(\A)\Big)
  \qquad
  \text{is decreasing.}
  \end{equation}
  Let 
  $$
  h_\mu(f,\A):= \lim_{N} \tfrac 1N\, 
  H_\mu\Big(\textstyle\bigvee_{n=0}^{N-1} f^{-n}(\A)\Big).
  $$
  The (metric) {\it entropy} of $\mu$ under $f$ is defined as
  $$
  h_\mu(f):=\sup\{\, h_\mu(f,\A)\;|\; \A \text{ is a Borel partition of $X$} \,\}.
  $$

  \begin{Theorem}[Variational Principle] (cf. Walters~\cite[Theorem~8.6]{Walters})
  \label{VarPrin}

  Let $f:X\to X$ be a continuous map of a compact metric space $X$. Then
  $$
  h_{top}(X,f)=\sup\{\, h_\mu(f)\;|\; \mu\in\cM(f)\,\}.
  $$
  \end{Theorem}

    \begin{Definition}\label{hexpan}\quad
    
  Let $f:X\to X$ be a homeomorphism. For $\e>0$ and $x\in X$ define
  $$
  \Ga_\e(x,f):=\{\,y\in X\;|\;\forall n\in \Z\quad  d(f^n(y),f^n(x))\le \e \,\}.
  $$
  We say that $f$ is {\it entropy expansive} or {\it h-expansive} if there
  is $\e>0$ such that
  $$
  \forall x\in X\qquad h_{\text{top}}(\Ga_\e(x,f),f)=0.
  $$
  Such an $\e$ is called an h-expansive constant for $f$.
    \end{Definition}
 
  The following theorems use the definition from Hurewicz and Wallman \cite{HuWall} of a finite dimensional 
  metric space. They mainly use the property from the Corollary to Theorem~V.1, page~55 in~\cite{HuWall},
  that if the metric space $X$ has dimension $\le m$, then for any $\ga>0$,  $X$ has a covering 
  $\mathfrak B(\ga)$ of diameter $<\ga$, such that  no point in $X$ is in more than $m+1$  
  elements of $\mathfrak B(\ga)$.

  \medskip

  \begin{Theorem}[Bowen~{\cite[Theorem~3.5]{Bowen9}}]\label{TBhE}\quad
  
  Let $X$ be a finite dimensional metric space and $f:X\to X$ a 
  uniformly continuous homeomorphism. Suppose that $\e>0$ is
  an h-expansive constant for $f$. 
  If $\A$ is a finite Borel partition of $X$ with $\diam\A<\e$
  then $h_\mu(f)=h_\mu(f,\A)$.
  \end{Theorem}

\medskip
  
  \begin{Definition}\label{dexpans}\quad
  
  A continuous flow $\phi:\re\times X\to X$ is said {\it flow expansive}
  (cf.~\cite{BoWa})
  if for every $\eta>0$ there exists $\de>0$ such that if
  $x,\,y\in X$ and $\a:\re\to\re$ is continuous with $\a(0)=0$
  such that $\forall t\in\re$ \;\; $d(\phi_t(x),\phi_{\a(t)}(y))\le \de$,
  then $y=\phi_v(x)$ for some $|v|\le\eta$.
  Observe that if a flow $\phi$ is flow expansive then its time~$1$ map $\phi_1$
  is entropy expansive.
  \end{Definition}
 
    For a continuous flow $\phi:\re\times X\to X$, the {\it metric entropy} 
  $h_\mu(\phi)$ of a $\phi$-invariant Borel probability measure $\mu$
  is defined as the metric entropy of its time one map
  $h_\mu(\phi):=h_\mu(\phi_1)$.

   \begin{Definition}\label{duhe}\quad
   
   Let $\cU$ be a topological subspace of $C^0(X,X)\supset\cU$ and 
   $Y\subseteq X$ compact.
   We say that $\cU$ is {\it uniformly h-expansive} on $Y$
   if there is $\e>0$ such that 
   $$
   \forall f\in\cU\quad \forall y\in Y
   \qquad
   h_{\text{top}}(\Ga_\e(y,f),f)=0.
   $$
   In our applications $\cU$ will be a $C^1$ neighbourhood of a diffeomorphism
   endowed with the $C^0$ topology. An h-expansive homeomorphism corresponds
   to $\cU=\{ f\}$.
   \end{Definition}

  Let $\cP(X)$ be the set of Borel probability measures on $X$ endowed 
  with the weak* topology. Let $\cM(f)\subset \cP(X)$ the subspace of $f$-invariant
  probabilities.

 \medskip
 
 \begin{Theorem}\quad\label{ESS}
 
Let $X$ be a finite dimensional compact metric space, $Y\subseteq X$ compact  and 
let 
\linebreak 
$\cU\subset C^0(X,X)$ be a uniformly h-expansive set on $Y$.
 Then the entropy map
$(\mu,f)\mapsto h_\mu(f)$ is upper semicontinuous on $\cU$,
i.e.

if $f\in\cU$, $\mu\in\cM(f)$ and $\supp\mu\subset Y$ then
given $\e>0$ there are open sets $\cV$, $U$,
\linebreak
$f\in\cV\subset C^0(X,X)$ and $\mu\in U\subset \cP(Y)$ such that 
$$
\forall g\in\cU\cap\cV,\quad
\forall\nu\in\cM(g)\cap U\cap\cP(Y)
\qquad 
h_\nu(g)\le h_\mu(f)+\e.
$$

In particular, this applies to time-one maps of uniformly expansive flows
as in section~\ref{suexpan} or Remark~\ref{rue}, giving
$$
\limsup_{(\psi^n_1,\nu_n)\to(\phi_1,\mu)}h_{\nu_n}(\psi^n)\le h_\mu(\phi).
$$
\end{Theorem}

\begin{proof}
\quad

Let $\e>0$ be a uniform entropy h-expansivity constant on $Y$ for all $f\in\cU$.
Fix $f\in\cU$ and $\mu\in\cM(f)$ and let $\de>0$. 

Let $\C=\{\, C_1,\ldots, C_n\,\}$ be a finite partition of $Y$
by Borel sets with $\diam C_i<\e$. By Theorem~\ref{TBhE},
$h_\nu(g)=h_\nu(g,\C)$ for all $g\in\cU$ and $\nu\in\cM(g)$.
Let $N$ be such that 
\begin{equation}\label{fixN}
\frac 1N\, H_\mu\Big(\textstyle\bigvee\limits_{k=0}^{N-1} f^{-k}\C\Big)
< h_\mu(f)+\tfrac 12\de.
\end{equation}
Since $\mu$ is regular, there are compact sets
$$
K_{i_0\ldots i_{N-1}}\subset\textstyle\bigcap\limits_{k=0}^{N-1} f^{-k}C_{i_k}
$$
such that
\begin{equation}\label{muck}
\mu\Big(\bigcap_{k=0}^{N-1} f^{-k}C_{i_k}\setminus K_{i_0\cdots i_{N-1}}\Big)<\e_1.
\end{equation}
Then
$$
L_j:=\textstyle \bigcup\limits_{k=0}^{N-1}\bigcup\limits_{i_k=j}
f^k K_{i_0\ldots i_{N-1}}
\subset C_j.
$$
The sets $L_1,\ldots, L_n$ are compact and disjoint so there is 
a partition $\D:=\{ D_1,\ldots , D_n \}$ with $\diam(D_j)<\e$ and 
$L_j\subset \intt(D_j)$.
We have that 
$$
K_{i_0\ldots i_{N-1}}\subset \intt\Big(\textstyle\bigcap\limits_{k=0}^{N-1}
f^{-k} D_{i_k}\Big).
$$
Choose open subsets $W_{i_0\ldots i_{N-1}}$ such that
$$
K_{i_0\ldots i_{N-1}}\subset 
W_{i_0\ldots i_{N-1}}\subset
\ov{W_{i_0\ldots i_{N-1}}} \subset
\intt\Big(\textstyle\bigcap\limits_{k=0}^{N-1}
f^{-k} D_{i_k}\Big).
$$
We have that 
$$
f^k\big(\ov{W_{i_0\ldots i_{N-1}}} \big)
\subset \intt D_{i_k}.
$$
Choose a relatively open subset $f\in\cU_1\subset \cU$ such that 
$$
\forall g\in\cU_1\quad \forall (i_0,\ldots, i_{N-1})
\quad
\forall k\in\{0,\ldots,N-1\}
\qquad
g^k\big(\ov{W_{i_0\ldots i_{N-1}}} \big)
\subset \intt D_{i_k}.
$$
So that 
$$
\forall g\in\cU_1\qquad
K_{i_0\ldots i_{N-1}}\subset 
W_{i_0\ldots i_{N-1}}\subset
\intt\Big(\textstyle\bigcap\limits_{k=0}^{N-1}
g^{-k} D_{i_k}\Big).
$$
By Urysohn's Lemma there exist $\psi_{i_0\ldots i_N}\in C^0(X,\re)$ such that 
\begin{itemize}
\item $0\le \psi_{i_0\ldots i_N}\le 1$.
\item equals $1$ on $K_{i_0\ldots i_{N-1}}$.
\item vanishes on $X\setminus W_{i_0\ldots i_{N-1}}$.
\end{itemize}
Define
$$
U_{i_0\ldots i_{N-1}}:=\Big\{ m\in\cP(Y)\;:\;
\Big|\int \psi_{i_0\ldots i_{N-1}}\,dm-\int \psi_{i_0\ldots i_{N-1}}\,d\mu\Big|
<\e_1 \Big\}.
$$
The set $U_{i_0\ldots i_{N-1}}$ is open in $\cP(Y)$
and if $m\in U_{i_0\ldots i_{N-1}}$ and $g\in\cU_1$ then
\begin{align}\label{mDk}
m\Big({\textstyle\bigcap\limits_{k=0}^{N-1}} g^{-k} D_{i_k}\Big)
\ge \int \psi_{i_0\ldots i_{N-1}}\,dm
>\int \psi_{i_0\ldots i_{N-1}}\,d\mu-\e_1
\ge \mu\big(K_{i_0\ldots i_{N-1}}\big) -\e_1.
\end{align}
From~\eqref{muck} and~\eqref{mDk} we get that
$$
\mu\Big(\textstyle\bigcap\limits_{k=0}^{N-1}f^{-k}C_{i_k}\Big)
-m\Big(\textstyle\bigcap\limits_{k=0}^{N-1}g^{-k}D_{i_k}\Big)
< 2 \e_1.
$$
If $U:=\bigcap_{i_0\ldots i_{N-1}}U_{i_0\ldots i_{N-1}}$ and $m\in U$  and
$g\in\cU_1$
 then:
$$
\Big|\mu\Big(\textstyle\bigcap\limits_{k=0}^{N-1} f^{-k}C_{i_k}\Big)
-m\Big(\bigcap\limits_{k=0}^{N-1}g^{-k}D_{i_k}\Big)\Big|
< 2\e_1 n^N,
$$
because if $\sum_{i=1}^q a_i=1=\sum_{i=1}^qb_i$ and there exists
$c>0$ such that $a_i-b_i<c$ for all $i$, then $\forall i$ $|a_i-b_i|<c\,q$,
because $b_i-a_i=\sum_{j\ne i}(a_j-b_j)<c\,q$.

So if $m\in U$, $g\in\cU_1$ and $\e_1$ is small enough, for $N$ fixed in \eqref{fixN}, the continuity of 
$x\,\log x$ gives:
$$
\frac 1N\,H_m\Big({\textstyle\bigvee\limits_{k=0}^{N-1}} g^{-k}\D\Big)
<
\frac 1N\,H_\mu\Big({\textstyle\bigvee\limits_{k=0}^{N-1}} f^{-k}\C\Big)
+\frac\de 2.
$$
Hence, for $g\in\cU_1$, $m\in U\cap\cM(g)$  and $\e_1$ small enough,
using  Theorem~\ref{TBhE} and \eqref{HfNdec}  we have that
\begin{align*}
h_m(g)&=h_m(g,\D)
\le \frac 1N\,H_m\Big(\textstyle\bigvee\limits_{k=0}^{N-1} g^{-k}\D\Big)
\\
&\le \frac 1N\,H_\mu\Big({\textstyle\bigvee\limits_{k=0}^{N-1}} f^{-k}\C\Big)
+\frac\de 2
< h_\mu(f)+\de.
\end{align*}
\end{proof}

   \bigskip

  \section{Symbolic Dynamics.}\label{ASD}
  
  Let $\cA:=\{T_1,\ldots,T_M\}$ be a finite set, 
  called the set of symbols, and let
  $$
  \cA^\Z =\textstyle\prod\limits_{n\in\Z}\cA
  =\big\{\,(x_i)_{i\in\Z}\;\big|\;\forall i\in\Z\;\;\; x_i\in \cA\,\big\}.
  $$
  We denote $\ox=(x_i)_{i\in\Z}$ and $\ox_i =x_i$.
  Endow $\cA$ with the discrete topology and $\cA^\Z$ with the 
  product topology. By Tychonoff Theorem $\cA^\Z$ is compact. 
  Given $a>1$ the metric 
  \begin{equation*}
  d_a(\ox,\oy)= a^{-n}, 
  \qquad 
  n=\max\{\, k\in\na \; |\; \forall|i|\le k,\quad x_i=y_i\,\}  
  \end{equation*}
  induces the same topology.
  
  The {\it shift} homeomorphism $\si:\cA^\Z\to\cA^\Z$ is defined by
  $\si(\ox)_i=x_{i+1}$. A subset $\Om\subset\cA^\Z$ is called a {\it subshift}
  if $\Om$ is closed and $\si(\Om)=\Om$. 
  We call $\Om$ a {\it subshift of finite type} iff  
  there is a function
  $A:\cA\times\cA\to\{0,1\}$ (or equivalently a matrix $A\in\{0,1\}^{M\times M}$)
  such that
  $$
  \Om=\Si(A):=\{\,\ox\in \cA^\Z\;|\; \forall i\in\Z \quad A(x_i,x_{i+1})=1\;\}.
  $$
  
  Suppose that $\Om$ is a subshift and $\tau:\Om\to\re^+$ is a positive continuous function. 
  The {\it suspension} $S(\Om,\si,\tau)$ is defined as the topological quotient
  space $S(\Om,\si,\tau)=\Om\times\re/_\equiv$, where
  \begin{equation}\label{som}
  \forall (x,s)\in\Om\times\re\qquad \big(x, \,s+\tau(x)\big) \equiv \big(\si(x), s\big).
  \end{equation}
  Equivalently,
  $$
  \forall (x,s)\in\Om\times\re \quad \forall n\in\Z \qquad
  \Big(x,\; s+\tsum_{i=0}^{n-1}\tau(\si^i(x))\Big)
  \equiv\big(\si^n(x),s\big).
  $$
  Then $S(\Om,\si,\tau)$ is a compact metrizable space. A metric appears in \cite{BoWa}.
  
  We obtain the  {\it suspension} flow $S_t=sus_t(\Om,\si,\tau)$ by ``flowing vertically'' and remembering identifications, i.e.
  \begin{equation}\label{sust}
  S_t(P(x,s)) = P(x,s+t),
   \end{equation}
  where $P:\Om\times\re \to S(\Om,\si,\tau)$ is the canonical projection. 
  
  \begin{Definition}\label{DHSF}\quad
  
  A {\it hyperbolic symbolic flow} is a suspension flow $sus_t(\Si_A,\si,\tau)$
   on $S(\Si_A,\si,\tau)$,
  where $\Si_A$ is a subshift of finite type and $\tau:\Si_A\to\re$ is positive and
  Lipschitz with respect to the metric $d_a$ for some $a>1$.
  \end{Definition}

     \bigskip

   \bigskip
 
  \section{Uniform Expansivity and Shadowing.}\label{asha}
  
  Let $\phi$ be the flow of a $C^1$ vector field on a compact manifold $M$.
  A compact $\phi$-invariant subset  $\La\subset M$
  is {\it  hyperbolic } for $\phi$ if the tangent bundle 
  restricted to $\La$ is decomposed as the Whitney sum
  $T_\La M = E^s \oplus E\oplus E^u$, where $E$ is the 1-dimensional 
  vector bundle tangent to the flow and there are constants $C,\la>0$ such that
  \begin{enumerate}[(a)]
  \item $D\phi_t(E^s) = E^s$, $D\phi_t(E^u)=E^u$ for all $t\in\re$
  \item\label{hipb} $| D\phi_t(v)| \le C\,\ee^{-\la t} |v|$ for all $v\in E^s$, $t\ge 0$.
  \item\label{hipc} $|D\phi_{-t}(u)|\le C\, \ee^{-\la t} |u|$ for all $u\in E^u$, $t\ge 0$.
  \end{enumerate} 
  It follows from the definition that the hyperbolic splittig 
  $E^s \oplus E\oplus E^u$ over $\La$ is continuous.

  From now on we shall assume that $\La$ does not contain fixed points
  for $\phi$.
  For $x\in \La$ define the following   stable and unstable sets:
  \begin{align}
  W^{ss}(x):&=\{\,y\in M\;|\; d(\phi_t(x),\phi_t(y))\to 0 \text{ as }t\to +\infty\,\},
  \notag\\
  W^{ss}_\e(x):&=\{\,y\in W^{ss}(x)\;|\;d(\phi_t(x),\phi_t(y))\le\e\;\;\forall t\ge 0\,\},
  \notag\\
  W^{uu}(x):&=\{\,y\in M\;|\; d(\phi_{-t}(x),\phi_{-t}(y))\to 0\text{ as } t\to +\infty\,\},
  \notag\\
  W^{uu}_\e(x):&=\{\,y\in W^{uu}(x)\;|\;d(\phi_{-t}(x),\phi_{-t}(y))\le \e\;\;\forall t\ge 0\,\},
  \label{wuue}\\
  \intertext\quad
  W^s_\e(x):&=\{\,y\in M\;|\;d(\phi_t(x),\phi_t(y))\le\e\;\;\forall t\ge 0\,\},
  \notag\\
   W^{u}_\e(x):&=\{\,y\in M\;|\;d(\phi_{-t}(x),\phi_{-t}(y))\le \e\;\;\forall t\ge 0\,\}.
   \notag
  \end{align}
  
  \medskip 
  
  Conditions~\eqref{hipb} and ~\eqref{hipc} are equivalent to
  \begin{enumerate}[(a)]
  \addtocounter{enumi}{3}
  \item\label{hipd} There exists $T>0$ such that
   $\lV D\phi_T|_{E^s}\rV <\tfrac 12$ \quad and \quad
   $\lV D\phi_{-T}|_{E^u}\rV<\tfrac 12$.
  \end{enumerate}

  \subsection{Uniform expansivity.}\quad

  Let $\fX^k(M)$ be the Banach manifold of the $C^k$ vector fields on $M$, $k\ge 1$.
  From now on $X\in\fX^k(M)$, $\phi_t$ is the flow of $X$ and $\La$ is a hyperbolic set for $\phi_t$.

  \begin{Proposition}\label{unifhip}\quad
  
  There are open sets $X\in\cU\subset\fX^1(M)$ and $\La\subset U\subset M$ 
  such that for every $Y\in\cU$ the set $\La_Y:=\bigcap_{t\in\re}\psi^Y_t(\ov U)$
  is hyperbolic for the flow $\psi^Y_t$ of $Y$, with uniform constants $C$, $\la$, $T$
  on~\eqref{hipb}, \eqref{hipc} and \eqref{hipd}.
  \end{Proposition}
  
   Proposition~\ref{unifhip} can be proven by a characterization of hyperbolicity using
   cones (cf. Fisher-Hasselblatt~\cite[Proposition~5.1.7]{FH}, Hasselblatt-Katok~\cite[Proposition~17.4.4]{HK}) and obtaining uniform 
   contraction (expansion) for a fixed iterate in $\La_Y$.

  \medskip

  \begin{Proposition}[Hirsch, Pugh, Shub {\cite[Corollary 5.6, p.~63]{HPS}}, 
  Bowen~{\cite[Prop. 1.3]{Bowen6}}]
  \label{pHPS}\quad

  There are constants $C,\,\la>0$ such that, for small $\e$,
  \begin{enumerate}[(a)]
  \item $d\big(\phi_t(x),\phi_t(y)\big)\le C\, \ee^{-\la t}\, d(x,y)$
  when $x\in \La$, $y\in W_\e^{ss}(x)$, $t\ge 0$.
  \item $d\big(\phi_{-t}(x),\phi_{-t}(y)\big)\le C\, \ee^{-\la t}\, d(x,y)$
  when $x\in \La$, $y\in W_\e^{uu}(x)$, $t\ge 0$.

  \end{enumerate} 
  \end{Proposition}
  
  \medskip
  \begin{mysec}{\bf Canonical Coordinates \cite[(3.1)]{PS}, \cite[(4.1)]{HPS}, \cite[(7.4)]{Smale},
  \cite[(1.4)]{Bowen6}, \cite[(1.2)]{Bowen3}:}
  \label{caco}
  \end{mysec}\vskip -5pt
  {\it There are $\de, \,\ga>0$ for which the following is true:
  If $x, y\in \La$ and $d(x,y)\le \de$ then there is a unique 
  $v=v(x,y)\in \re$ with $|v|\le \ga$ such that 
  \begin{equation}\label{ecaco}
  \langle x,y\rangle:=W_\ga^{ss}(\phi_v(x))\cap W^{uu}_\ga(y) \ne \emptyset.
  \end{equation}
  This set consists of a single point, which we denote 
  $\langle x,y\rangle\in M$. The maps $v$ and 
  $\langle\;,\;\rangle$ are continuous on the set
  $\{\,(x,y)\;|\; d(x,y)\le \de\,\}\subset \La\times\La$.
  }
  
  \medskip
  
     We will take a small neighborhoods $\La\subset U\subset M$ and  $\cU\subset\fX^1(M)$
     and  take uniform constants from \ref{pHPS} and \ref{caco} which hold for every $Y\in\cU$ and
     all points in the maximal invariant set $\La^Y_U:=\bigcap_{t\in\re}\psi^Y_t(\ov U)$.
    The following proposition is a modification of 
    Bowen {\cite[Prop.~1.6, p.~4]{Bowen6}} which we prove below.
  
    \begin{Proposition}
    \label{B16}
  \quad
  
     There are open sets $X\in\cU\subset\fX^1(M)$
     and $\La\subset U\subset M$ and $\eta_0, \ga>0$, $B>1$ such that
     $$
     \forall \eta>0 \qquad \exists
      \be=\be(\eta)=\tfrac 1B\,\min\{\eta,\eta_0\} \qquad \forall Y\in\cU
     $$
   if $\psi_t=\psi^Y_t$  is the flow of $Y$, 
  $x,\, y \in \Om^Y_U:=\bigcap_{t\in\re}\psi_t(\ov U)$ and $s:\re\to\re$  continuous with $s(0)=0$ satisfy
  \begin{equation}\label{csh}
  d(\psi_{t+s(t)}(y),\psi_t(x))\le\be\quad\text{ for }|t|\le L,
  \end{equation}
  then 
  \begin{equation}\label{setav}
  |s(t)|\le3\eta \quad \text{ for all }|t|\le L,
  \qquad
   |v(x,y)|\le\eta \quad \text{ and }
  \end{equation}
    \begin{align}
  \forall |s|\le L, \qquad
  &d(\psi_s(y),\psi_{s+v}(x))\le C\,\ee^{-\la(L-|s|)}\,
  \big[d(\psi_L(w),\psi_L(y))+d(\psi_{-L}(w),\psi_{-L+v}(x))\big],
  \notag\\
  &\text{where }\qquad w:=\langle x,y\rangle 
  =W^{ss}_\ga(\psi_v(x))\cap W^{uu}_\ga(y).
  \label{d<egdd1}
  \end{align}
  also
  \begin{equation}\label{dsh}
  \forall |s|\le L, \qquad
   d(\psi_s(y),\psi_s\psi_v(x))\le C\, \ga\,e^{-\la (L-|s|)}.
  \end{equation}
   In particular 
  $$
  d(y,\psi_v(x))\le C\, \ga\, e^{-\la L}.
  $$
  \end{Proposition}
  
  \bigskip
  
  For the proof of Proposition~\ref{B16} we need the following
  
    \begin{Lemma}\label{B4}

  There is $\eta_0>0$ and $B>1$ such that  
  \newline
  if $d(x,y)\le\eta_0$,  
  $Y\in\cU$, $x,\,y\in\La^Y_U$ and $\eta = B\, d(x,y)$  then
  \begin{gather}
  \langle x,y\rangle\in W^{ss}_\eta(\psi^Y_v(x))\cap W^{uu}_\eta(y)
  \qquad \text{with }\quad
  |v(x,y)| \le \eta \qquad 
  \label{vdxy}\\
  \text{and } \qquad d(x,\psi^Y_v(x))\le \eta.
  \label{Bdxy}
  \end{gather}
  \end{Lemma}

 \begin{proof}\quad

We have that $\langle x,x\rangle =x$ and $v(x,x)=0$.
By uniform continuity, given $\de>0$, for $d(x,y)$ small enough
\begin{equation}\label{dewx}
d(\langle x,y\rangle,x)\le \de,
\qquad
d(\langle x, y\rangle,y)\le \de,
\end{equation}
and $v=v(x,y)$ is so small that
\begin{equation}\label{dpsvd}
d(\psi_v(x),x)\le \de.
\end{equation}

 The continuity of the hyperbolic splitting implies that the angles
  $\measuredangle(E^s,E^u)$, $\measuredangle(Y,E^s)$ and 
  $\measuredangle(E^s\oplus\re Y,E^u)$
   are bounded away from zero, 
  uniformly on $\La^Y_V$, for some $V\supset U$
   and all $Y$ in an open set $\cU_0\subset\fX^1(M)$
  with $X\in\cU_0$.   
    There is $\be_1>0$ such that if $x,\,y\in\La^Y_U$ and $d(x,y)<\be_1$
  then 
  $$
  \langle x, y\rangle =W^s_\ga(x)\cap W^{uu}_\ga(y)\in V.
  $$

   The strong local invariant manifolds $W^{ss}_\ga$, $W^{uu}_\ga$ 
  are tangent to $E^s$, $E^u$ at $\La^Y_V$ and for a fixed $\ga$ as $C^1$ submanifolds
  they vary continuously on the base point $x\in M$ and on the vector field in $C^1$
  topology (cf.  \cite[Thm. 4.3]{CroPo}\cite[Thm. 4.1]{HPS}). 
  There is a family of small cones $E^u_X(x)\subset C^u(x)\subset T_xM$, $E^s_X(x)\subset C^s(x)\subset T_xM$ 
  defined on a neighbourhood $W$ of $\La$ such that $\exp_x(C^u(x)\cap B_\de(0))$, 
  $\exp_x (C^s(x)\cap B_\de(0))$ are invariant under $\psi^Y_1$ and $\psi^Y_1$ respectively, for
  $Y$ in a $C^1$ neighborhood $\cW$ of $X$. These cones contain $W^{uu}_{\ga}(x)$ and $W^{ss}_{\ga}(x)$ for 
  $x\in\La^Y_W$ and $Y\in\cW$. The angles between these cones are uniformly bounded 
  away form zero, so for example if $z^u\in W^{uu}(x)$, $z^s\in W^{ss}(x)$ and $d(z^u,x)$, $d(z^s,x)$ are small,
  then $d(z^u,x)+d(z^s,x)< A_0\,d(z^u, z^s)$ for some $A_0>0$.
  We can construct similiar cones separating $E^u$ from $E^s\oplus\re X$.
  
  Shrinking $U$ and $\cU$ if necessary there are $0<\be_2<\be_1$  and $A_1,\,A_2,\,A_3>0$
  such that if $Y\in\cU$, $x,\,y\in \La^Y_U$ and $d(x,y)<\be_2$,
  taking 
  $w:=\langle x,y\rangle\in W^s_\ga(x)\cap W^{uu}_\ga(y)$
  and $v$  such  that $w\in W^{ss}_\ga(\psi^Y_v(x))$,  i.e. 
$\psi_v^Y(x)\in\psi^Y_{[-1,1]}(x)\cap W^{ss}_\ga(w)$, then
  \begin{align}
d(x,w)+d(w,y)&\le A_1\, d(x,y),
 \label{a1xwy}
\\ 
d(x,\psi^Y_v(x))+d(\psi^Y_v(x),w)
&\le A_2\, d(x,w) \le A_2 A_1 \,d(x,y),
\label{a2xpw}
\\
|v|\le A_3\, d(x,\psi^Y_v(x))&\le A_3 A_2 A_1 \, d(x,y).
\notag
  \end{align}

  We can assume that $\cU_0$ and $U$ are so small that the constants
  $C$, $\la$, $\e$ in Proposition~\ref{pHPS} can be taken uniform for all 
  $Y\in\cU_0$ and in $\La^Y_U$.
    By Proposition~\ref{pHPS},
since $w=\langle x,y\rangle \in W^{ss}_\ga(\psi^Y_v(x))$, 
 we have that
\begin{align*}
\forall t\ge 0 \qquad
d\big(\psi^Y_t(\langle x,y\rangle),\psi^Y_{t}(\psi^Y_v(x))\big)
&\le C\,\ee^{-\la t}\,d(w,\psi_v^Y(x)) 
\\
&\le A_2 A_1 C\, \ee^{-\la t} \,d(x,y)
\qquad\text{ using~\eqref{a2xpw}.} 
\end{align*} 
Take $B_1:=(1+A_2)A_1 C$. Then if $d(x,y)<\be_2$ and 
$\eta= B_1\,d(x,y)$
 we obtain that
$\langle x,y\rangle\in W^{ss}_\eta(\psi^Y_v(x))$.

Since $\langle x,y\rangle\in W^{uu}_\ga(y)$ we have that 
\begin{align*}
\forall t\ge 0 \qquad
d(\psi^Y_{-t}(\langle x,y\rangle),\psi^Y_{-t}(y))
&\le C\, \ee^{-\la t} \,d(w,y) 
\\
&\le A_1C\, \ee^{-\la t} d(x,y)
\qquad\text{using }\eqref{a1xwy}.
\end{align*}
Thus if $\eta = B_1\, d(x,y)$ then $\langle x, y\rangle \in W^{uu}_\eta(y)$.

By~\eqref{dewx} and~\eqref{dpsvd}
there is $0<\be_0<\be_2$ such that if $d(x,y)<\be_0$ then
$d(w,x)$, $d(w,y)$ and $d(\psi_v(x),x)$ are small enough
to satisfy the above inequalities.
Now let 
$$
B:=\max\{ 1,\, B_1,\, A_3 A_2 A_1,\, A_2 A_1\}.
$$

 \end{proof}

  \noindent{\bf Proof of Proposition~\ref{B16}:}
  
  Let $\ga$ be 
  from~\ref{caco}.
  We may assume that $\eta$ is so small that
  \begin{gather}
  \eta < \tfrac{\ga}8,
  \label{etaga8}
  \\
  \sup\{\,d(\psi_u(x),x)\;:\; x\in M,\,|u|\le 4\eta\,\}\le \tfrac \ga 8.
  \label{u4eta}
  \end{gather}
  Let 
  \begin{equation}\label{beeta}
  \be=\be(\eta) := \tfrac 1B\, \min\{\eta,\eta_0\},
  \end{equation}
  where $B>1$ and $\eta_0$  are
  from Lemma~\ref{B4}. Consider $x$, $y$ and $s(t)$ as
  in the hypothesis. Since $s(0)=0$ we have that  $d(x,y)\le \be$.
  Using Lemma~\ref{B4} we can define
   \begin{equation}\label{wxyeta}
   w:=\langle x,y\rangle =W^{ss}_\eta(\psi_v(x))\cap W^{uu}_\eta(y)\ne\emptyset,
   \end{equation}
   we also have 
   \begin{equation}\label{vxyeta}
   |v|=|v(x,y)|\le \eta.
    \end{equation}
       Define the sets
   \begin{alignat*}{2}
   A&:=\{\,t\in[0,L]\;:\;|s(t)|\ge 3\eta\;&&\text{ or }
   \;d(\psi_t(y),\psi_t(w))\ge \tfrac 12{\ga}\,\},
   \\
   B&:=\{\,t\in[0,L]\;:\; |s(-t)|\ge 3\eta\;&&\text{ or }
    \;d(\psi_{-t+v}(x),\psi_{-t}(w))\ge \tfrac12\ga\,\}.
   \end{alignat*}
   
   Suppose that $A\ne \emptyset$. Let $t_1:=\inf A$. 
   Then 
   $d(\psi_t(y),\psi_t(w))\le \tfrac 12 \,\ga$, $\forall t\in[0,t_1]$.
   Since $w\in W^{uu}_\eta(y)$ and by~\eqref{etaga8}, $\eta<\tfrac 1{8}\ga$;
   from~\eqref{wuue} 
   we have that $d(\psi_t(y),\psi_t(w))\le \tfrac 18 \ga$, $\forall t\le 0$.
   Therefore
   \begin{equation}\label{wuut1}
   d(\psi_{t_1-r}(y),\psi_{t_1-r}(w))\le \tfrac 12 \,\ga, \qquad \forall r\ge 0.
   \end{equation}

   Since $s$ is continuous, $s(0)=0$ and $t_1\in\partial A$, we have that 
   $|s(t_1)|\le 3\eta$.
   Using~\eqref{u4eta} twice with $u=|s(t_1)|$, \eqref{wuut1} and the triangle inequality
   we obtain
   $$
   d(\psi_{t_1+s(t_1)-r}(y),\psi_{t_1+s(t_1)-r}(w))\le  \tfrac 34 \ga,
   \qquad \forall r\ge 0.
   $$
   Hence $\psi_{t_1+s(t_1)}(w)\in W^{uu}_\ga(\psi_{t_1+s(t_1)}(y))$. 
   From~\eqref{wxyeta},
   $w\in W^{ss}_\eta(\psi_v(x))$, and then using~\eqref{etaga8},
   \begin{equation}\label{uvga8}
   d(\psi_r(w),\psi_{r+v}(x))\le \eta<\tfrac \ga 8,
   \qquad \forall r\ge0.
   \end{equation}
   Since $|s(t_1)|\le 3\eta$, using~\eqref{u4eta} twice with $u=s(t_1)$, 
   and~\eqref{uvga8} with $r=t_1+p\ge 0$, and the triangle inequality,
   we get
   $$
   d(\psi_{t_1+s(t_1)+p}(w),\psi_{t_1+s(t_1)+v+p}(x))\
   \le\tfrac {3\ga}8, \qquad \forall p\ge 0.
   $$
   Hence $\psi_{t_1+s(t_1)}(w)\in W^{ss}_\ga(\psi_{s(t_1)+v}(\psi_{t_1}(x)))$.
   We have shown that
   \begin{equation}\label{pstst}
   \psi_{t_1+s(t_1)}(w)\in W^{ss}_\ga(\psi_{s(t_1)+v}(\psi_{t_1}(x)))
   \cap W^{uu}_\ga(\psi_{t_1+s(t_1)}(y)).
   \end{equation}
   Since $|s(t_1)+v|\le|s(t_1)|+|v|\le 4\eta<\ga$
   and by~\eqref{csh}, 
   \begin{equation}\label{lbe}
   d(\psi_{t_1+s(t_1)}(y),\psi_{t_1}(x))\le\be,
   \end{equation}
   equation~\eqref{pstst} implies that 
   \begin{gather*}
   v(\psi_{t_1}(x),\psi_{t_1+s(t_1)}(y))=s(t_1)+v(x,y),
   \\
   \psi_{t_1+s(t_1)}(w)=\langle \psi_{t_1}(x),\psi_{t_1+s(t_1)}(y)\rangle.
   \end{gather*}
   By Lemma~\ref{B4}, \eqref{lbe} and \eqref{beeta}, 
   \begin{gather}
   |s(t_1)+v|\le\eta \qquad \text{ and }
   \label{st1v}
   \\
   \psi_{t_1+s(t_1)}(w)\in W^{uu}_\eta(\psi_{t_1+s(t_1)}(y)),
   \;\text{in particular}
   \notag
   \\
   d(\psi_{t_1+s(t_1)}(w),\psi_{t_1+s(t_1)}(y))\le\eta.
   \label{dpstst}
   \end{gather}
   Since $|s(t_1)|\le 3\eta$, from~\eqref{u4eta},~\eqref{dpstst}
   and~\eqref{etaga8}, we get that
   $$
   d(\psi_{t_1}(w),\psi_{t_1}(y))\le \eta+2\left(\tfrac\ga 8\right)\le 
   \tfrac {3\ga}8.
   $$
   From~\eqref{st1v} and~\eqref{vxyeta} we have that
   $$
   |s(t_1)|\le|s(t_1)+v|+|v|\le 2\eta.
   $$
   These statements contradict $t_1\in A$.
   Hence $A=\emptyset$.
   
   Similarly one shows that $B=\emptyset$.
   Since $A=\emptyset$, inequality~\eqref{ywtga} holds for all $t\in[0,L]$.
   From~\eqref{wxyeta}, $w\in W^{uu}_\eta(y)$ and by~\eqref{etaga8}, 
   $\eta<\tfrac\ga 8$; thus inequality~\eqref{ywtga} also holds for $t\le 0$.
   \begin{equation}\label{ywtga}
   \forall t\le L\qquad d(\psi_t(y),\psi_t(w))< \tfrac 12{\ga}.
   \end{equation}
   Therefore
   \begin{equation}\label{psiLWuu}
    \psi_L(w)\in W^{uu}_{\frac 12\ga}(\psi_L(y)).
   \end{equation} 
   From Proposition~\ref{pHPS} we get
  \begin{equation*}
  \forall |s|\le L \qquad 
  d(\psi_s(w),\psi_s(y))\le C\,\ee^{-\la(L-|s|)}
  \,d(\psi_L(w),\psi_L(y)).
  \end{equation*}
  Similarly, $B=\emptyset$ imples that 
  \begin{equation}\label{psi-Lwss}
  \psi_{-L}(w)\in W^{ss}_{\frac 12\ga}(\psi_{-L+v}(x))
  \qquad \text{ and }
  \end{equation}
  $$
  \forall |s|\le L
  \qquad
  d(\psi_s(w),\psi_{s+v}(x))\le C\,\ee^{-\la(L-|s|)}\,
  d(\psi_{-L}(w),\psi_{-L+v}(x)).
  $$
   Adding these inequalities we obtain
  \begin{align}
  \forall |s|\le L \qquad
  &d(\psi_s(y),\psi_{s+v}(x))\le C\,\ee^{-\la(L-|s|)}\,
  \big[d(\psi_L(w),\psi_L(y))+d(\psi_{-L}(w),\psi_{-L+v}(x))\big],
  \notag\\
  &\text{where }\qquad w:=\langle x,y\rangle 
  =W^{ss}_\ga(\psi_v(x))\cap W^{uu}_\ga(y).
  \label{d<egdd}
  \end{align}
  This proves inequality~\eqref{d<egdd1}.
  
    From~\eqref{vxyeta}, $|v(x,y)|\le\eta$.
  The fact $A\cup B=\emptyset$ also gives $|s(t)|\le 3\eta$
  for $t\in[-L,L]$. This proves~\eqref{setav}.
  From~\eqref{psiLWuu}, \eqref{psi-Lwss} and~\eqref{d<egdd}
  we get inequality~\eqref{dsh}.
  
  \qed

   \begin{Proposition}\label{B5}\quad
   
   Let $\be(\eta)$ be from Proposition~\ref{B16}.
   \begin{enumerate}[(a)]
   \item\label{B5a}
   If $x,\,y\in\La$ and $s:[0,+\infty[\to\re$ continuous
   with $s(0)=0$
   satisfy 
   $$
   d(\psi_{t+s(t)}(y),\psi_t(x))\le \be \qquad \forall t\ge0,
   $$
   then $|s(t)|\le 3\eta$ for all $t\ge 0$ and there is $|v(x,y)|\le\eta$
   such that $y\in W^{ss}_\ga(\psi_v(x))$.
   
   \item\label{B5b}
   Similarly, if $x,\,y\in\La$, $s:]-\!\infty,0]\to\re$ is continuous with $s(0)=0$
   and
   $$
   d(\psi_{t+s(t)}(y),\psi_t(x))\le\be \qquad \forall t\le 0,
   $$
    then $|s(t)|\le 3 \eta$ for all $t\le 0$ and there is $|v(x,y)|\le\eta$
   such that $y\in W^{uu}_\ga(\psi_v(x))$.
   \end{enumerate}
   \end{Proposition}

  \begin{proof}\quad
  
  We only prove item~\eqref{B5a}.
  The same proof as in Proposition~\ref{B16} shows that 
  taking
  $$
  w:=\langle x,y\rangle=W^{ss}_\eta(\psi_v(x))\cap W^{uu}_\eta(y)\ne \emptyset,
  $$
  we have that
  $|v|=|v(x,y)|\le\eta$ and 
  $$
   \emptyset=A:=\{\,t\in[0,+\infty[\;:\;|s(t)|\ge 3\eta\;\text{ or }
   \;d(\psi_t(y),\psi_t(w))\ge \tfrac 12 {\ga}\,\}.
  $$
  Therefore $|s(t)|\le 3\eta$ for all $t\ge 0$ and 
  $w\in W^{ss}_{\frac 12 \ga }(y)\cap W^{ss}_\eta(\psi_v(x))$.
  Since $\tfrac 12 \ga+\eta<\ga$ we get that
  $y\in W^{ss}_\ga(\psi_v(x))$.

  \end{proof}

  \medskip

   \begin{Proposition}\label{B71}\quad

         There are $D>0$, $\be_0>0$ and  open sets $X\in\cU\subset\fX^1(M)$,
     $\La\subset U\subset M$,  such that
     $$
     \forall \be\in]0,\be_0] \qquad \forall Y\in\cU,
     $$
   if $Y\in\cU$, $\psi_t=\psi^Y_t$  is the flow of $Y$, 
  $x,\, y \in \La^Y_U:=\bigcap_{t\in\re}\psi_t(\ov U)$ 
  and $s:\re\to\re$  continuous with $s(0)=0$ satisfy
  \begin{equation}\label{csh2}
  d(\psi_{t+s(t)}(y),\psi_t(x))\le\be\quad\text{ for }|t|\le L,
  \end{equation}
  then $|s(t)|\le D\be$ for all $|t|\le L$ and there is $|v|=|v(x,y)|\le D\be$ 
  such that
     \begin{align*}
  \forall |s|\le L, \qquad
  d(\psi_s(y),\psi_{s+v}(x))\le D\,\be\,\ee^{-\la(L-|s|)}.
  \end{align*}
  
  Moreover for all $|s|\le L$,
   \begin{equation}\label{dextxy}
  d(\psi_s(y),\psi_{s+v}(x))\le
   D\,\ee^{-\la(L-|s|)}\,
  \big[ d(\psi_L(y),\psi_{L+v}(x)) + d(\psi_{-L}(y),\psi_{-L+v}(x)) \big],
  \end{equation}
  and $v$ is determined by
  $$
  \langle x,y\rangle =W^{ss}_\ga(\psi_v(x))\cap W^{uu}_\ga(y)\ne \emptyset.
  $$

  \end{Proposition}

    \begin{proof}\quad
  
  Let $C$, $\cU$, $U$ $\eta_0>0$ and $B$ be from Proposition~\ref{B16}.
  The continuity of the hyperbolic splitting implies that the angle
  $\measuredangle(E^s,E^u)$ is bounded away from zero.
  As in the argument after \eqref{dpsvd}, there are invariant families of cones separating
  $E^s$ from $E^u$ whose image under the exponential map contain the local invariant 
  manifolds $W^{ss}_\ga$, $W^{uu}_\ga$.
  And hence as in~\eqref{a1xwy}
  there are $A,\,\be_1>0$ such that if $x,\, y\in \La^Y_U$,
  $d(x,y)<\be_1$ and 
  $$
  w=\langle x,y\rangle =W^{ss}_\ga(\psi_v(x))\cap W^{uu}_\ga(y),
  $$
  then 
  \begin{equation}\label{wpsiv}
  d(w,\psi_v(x))+d(w,y) \le A\, d(\psi_v(x),y).
  \end{equation}
  Suppose that $0<\be<\min\{\tfrac 1B\eta_0,\,\be_1\}$
  and 
  $x$, $y$, $s(t)$, $\psi^Y_t$, $L$ satisfy~\eqref{csh2}.
  Apply Proposition~\ref{B16} with $\eta:= B\be$.
  
  Then $|s(L)|\le 3\eta$, 
  and 
  \begin{align*}
  d(\psi_L(y),\psi_L(x))
  &\le d(\psi_{L+s(L)}(y),\psi_L(x))+|s(L)| \cdot \Vert Y\Vert_{\sup}
  \\
  &\le \be + 3\eta \lV Y\rV_{\sup} < \de,
  \end{align*}
  if $\be$ is small enough, where $\de$ is from~\ref{caco}.
  So that $\langle \psi_L(x),\psi_L(y)\rangle$ is well defined.
  Similarly $|s(-L)|\le 3\eta$ and
  $d(\psi_{-L}(y),\psi_{-L}(x))<\de$.
  Since the time $t$ map $\psi_t$ preserves the family of
  strong invariant manifolds,
  in equation~\eqref{d<egdd1} we have that 
  \begin{align*}
  \psi_L(w) &=\langle \psi_{L}(x),\psi_L(y)\rangle =
  W^{ss}_\ga(\psi_{L+v}(x))\cap W^{uu}_\ga(\psi_L(y)),
  \\
    \psi_{-L}(w) &=\langle \psi_{-L}(x),\psi_{-L}(y)\rangle =
  W^{ss}_\ga(\psi_{-L+v}(x))\cap W^{uu}_\ga(\psi_{-L}(y)).
  \end{align*}
  Therefore, using~\eqref{wpsiv},
  \begin{align}
  d(\psi_L(w),\psi_L(y))+d&(\psi_{-L}(w),\psi_{-L+v}(x))
  \notag\\
  &\le  A \big[ d(\psi_{L+v}(x),\psi_L(y)) + d(\psi_{-L+v}(x),\psi_{-L}(y)) \big],
   \label{wxyd}
   \\
  d(\psi_{L+v}(x),\psi_{L}(y)) &\le
  d(\psi_{L+v}(x),\psi_{L}(x))+d(\psi_{L}(x),\psi_{L+s(L)}(y))
  +d(\psi_{L+s(L)}(y),\psi_{L}(y))
  \notag \\
  &\le |v| \lV Y\rV_{\sup}+\be+ |s(L)|\,\lV Y\rV_{\sup}
  \notag\\
  &\le B_1 \be,
  \notag
  \end{align}
  for some $B_1=B_1(\cU)>0$, because by Proposition~\ref{B16},
  $|v|\le \eta$, $|s(t)|\le 3\eta$ and $\eta = B \be$, so that
  $$
  |v| \le B\be,\qquad |s(t)|\le 3 B \be.
  $$
  A similar estimate holds for $d(\psi_{-L+v}(x),\psi_{-L}(y))$
  and hence from~\eqref{wxyd},
  $$
  d(\psi_L(w),\psi_L(y))+d(\psi_{-L}(w),\psi_{-L+v}(x))
  \le  2AB_1 \,\be.
  $$
  Replacing this  in~\eqref{d<egdd1} we have that
  $$
  \forall |s|\le L,\qquad
  d(\psi_s(y),\psi_{s+v}(x))\le D_1\,\be \,\ee^{-\la(L-|s|)},
  $$
  where $D_1=2 A B_1 C$.
  
  By~\eqref{wxyd} and~\eqref{d<egdd1} we also have that 
  $$
  d(\psi_s(y),\psi_{s+v}(x))\le
  AC\,\ee^{-\la(L-|s|)}\,
  \big[ d(\psi_L(y),\psi_{L+v}(x)) + d(\psi_{-L}(y),\psi_{-L+v}(x)) \big].
  $$
  Now take $D:=\max\{ D_1,\,B,\,3B,\,AC\,\}$.
  
  \end{proof}

  \begin{Definition}\label{dfe}\quad
  
  We say that $\psi|_\La$ is {\it flow expansive} if for every 
  $\eta>0$ there is $\ov\a=\ov\a(\eta)>0$ such that 
  if $x\in\La$, $y\in M$ and  there is $s:\re\to\re$ continuous
  with $s(0)=0$ and $d(\psi_{s(t)}(y),\psi_t(x))\le\ov \a$ for all $t\in\re$,
  then
  $y=\psi_v(x)$ for some 
  $|v|\le \eta$.
  \end{Definition}

  \begin{Remark}\label{rue}\quad
  
  Observe that Proposition~\ref{B16} implies 
  uniform expansivity in a neighbourhood of $(X,\La)$, namely
  there are neighbourhoods $X\in\cU\subset\fX^1(M)$ and $\La\subset U\subset M$
  such that for every $\eta>0$ there is $\a=\a(\eta,\cU,U)>0$ such that if
  $x\in\La^Y_U:=\cap_{t\in\re}\psi^Y_t(\ov U)$, $y\in M$, $s:(\re,0)\to (\re,0)$ 
  continuous and $\forall t\in \re$, $d(\psi^Y_{s(t)}\big(y),\psi^Y_t(x)\big)<\a$;
  then $y=\psi^Y_v(x)$ for some $|v|<\eta$.
  
  This also implies uniform h-expansivity as in Definition~\ref{duhe}.
  \end{Remark}
  
  \medskip
  
  \begin{Corollary}[\bf Uniform expansivity and  uniform h-expansivity]\label{Rfe}\quad
  
  There are open sets $X\in\cU\subset\fX^1(M)$, $\La\subset U\subset M$
  such that 
  
  \noindent 
  for all $Y\in\cU$ the set $\bigcap_{t\in\re}\psi^Y_t(\ov U)$ is hyperbolic for 
  $\psi^Y_t$ with uniform hyperbolic constants $C,\,\la$ on $\cU$ and
  $$
  \forall \eta>0 \qquad \exists \a=\a(\eta)>0 \qquad \forall Y\in\cU
  $$
  if $\psi^Y_t$ is the flow of $Y$, \;$x,y\in\bigcap_{t\in\re}\psi^Y_t(\ov U)$, 
  $s:\re\to\re$ is continuous, $s(0)=0$ and
  $$
  \forall t\in\re \qquad 
  d\big(\psi^Y_{s(t)}(y),\psi^Y_t(x)\big)\le \a,
  $$
  then $y=\psi_v(x)$ for some $|v|\le \eta$.
  \end{Corollary}

  \subsection{Shadowing.}\quad

  \bigskip

  \begin{Definition}\label{B8}\quad
  
  Let $L>0$, we say that $(T,\Ga)$ is an $L$-specification if
  \begin{enumerate}[(a)]
  \item $\Ga=\{x_i\}_{i\in\Z}\subset \La$.
  \item $T=\{t_i\}_{i\in\Z}\subset\re$\quad and\quad  $t_{i+1}-t_i\ge L$\; $\forall i\in\Z$.
  \end{enumerate}
  We say that the specification $(T,\Ga)$ is $\de$-possible if
  $$
  \forall i\in\Z\qquad d(\psi_{t_i}(x_i),x_{i+1})\le \de.
  $$ 
  \end{Definition}
  
  If $s:\re\to\re$ we denote
  \begin{align*}
  U_\e(s,T,\Ga):&=\big\{\,y\in M\,\big|\, d(\psi_{t+s(t)}(y),\psi_t(x_i))\le\e\quad\text{for }
  t\in]t_i,t_{i+1}[\,\big\};
  \\
  STEP_\e(T):&= \big\{\, s \;\big|\; s|_{]t_i,t_{i+1}[} \text{ is constant,}\;
  s(t_i)\in\{ s(t_i-),\,s(t_i+)\}, 
  \\
  &\hskip 2cm |s(t_0)|\le \e \text{ and } |s(t_i+)-s(t_i-)|\le \e\;\big\};
  \\
  U^*_\e(T,\Ga):&=\bigcup \big\{\,U_\e(s,T,\Ga)\;|\;s\in STEP_\e(T)\,\big\}.
  \end{align*}
  
  If $y\in U^*_\e(T,\Ga)$ we say that  the point $y$ $\e$-{\it shadows} the specification
  $(T,\Ga)$. 
  
  \bigskip

  \begin{Remark}\label{RSHL}\quad
  \begin{enumerate}[(a)]
  \item
  Observe that a function $s\in STEP_\e(T)$ is possibly discontinuous. But
  from the conditions in $STEP_\e(T)$ and $U_\e(s,T,\Ga)$ it is easy to
  replace $s$ by a continuous function satisfying~\eqref{csh} with $\be= K \e$.
  Indeed, replace $s$ by
  \begin{equation}\label{sie1}
  \si(t) = s(t_i-\e)+ \frac{s(t_i+\e)-s(t_i-\e)}{2\e}\big( t-(t_i-\e)\big)
  \quad\text{if}\quad t\in[t_i-\e,t_i+\e],
  \end{equation}
  and $\si(t)=s(t)$ otherwise. Then
  \begin{align}
  d(\psi_{t+\si(t)}(y),\psi_{t+s(t)}(y))&\le \lV\partial_t\psi\rV_{\sup} |\si(t)-s(t)|
  \le  \lV\partial_t\psi\rV_{\sup} \e,
  \notag\\
   d(\psi_{t+\si(t)}(y),\psi_t(x_i))&\le \e +  \lV\partial_t\psi\rV_{\sup} \e
   = K\,\e \qquad \text{ when }\quad |t-t_i|\le\e.
   \label{Ke1}
   \end{align}
   \item  In~\eqref{sie1} the continuous function $t+\si(t)$ is strictly increasing.
   Indeed, $s(t)$ is constant on each interval $]t_i,t_{i+1}[$ and
   $|s(t_{i}+)-s(t_i-)|\le \e$, therefore
   $$
   |\si'(t)| =\lv\tfrac{s(t_i+)-s(t_i-)}{2\e}\rv\le\tfrac 12 
   \quad\text{ on } t\in[t_i-\e,t_i+\e], 
   \quad \si'(t)=0 \text{ otherwise}.
   $$
   And thus
   $$
   \tfrac{d\,}{dt}(t+\si(t))\ge 1-\tfrac 12 >0, \qquad t\ne t_i.
   $$
   \item\label{B7c}
    Similarly we can modify the function s by a function $\si$ which 
   is continuous, strictly increasing, satisfying \eqref{Ke1} for some $K$ 
   independent of $\e$ and also $\si(t_0)=0$.
   Indeed, define
  $$
  \si_2(t) =
  \begin{cases}
   \frac1{3\e}(t-t_0) \,s(t_0+3\e) &\text{if} \quad t-t_0\in[0,3\e],
   \\
   \frac 1{3\e}(t_0-t) \, s(t_0-3\e) &\text{if} \quad t-t_0\in[-3\e,0],
   \\
   \si(t) &\text{if} \quad t-t_0\notin[-3\e,3\e].
  \end{cases}
  $$
  Then $\si_2(t_0)=0$.
  Since $|s(t_0)|\le \e$ and $|s(t_0+)-s(t_0-)|\le\e$, if $\e<L/3$ we have that
 $|s(t_0+3\e)|=|s(t_0+)|\le 2\e$ and  $|s(t_0-3\e)|=|s(t_0-)|\le 2\e$. Therefore
 $|\si_2'(t)|\le \frac23$ and then
 $$
 \tfrac{d\,}{dt}(t+\si_2(t))\ge 1 -\tfrac 23 >0
 \qquad \text{if }\quad 0<|t-t_0|\le 3\e.
 $$
 Also $|\si_2(t)-s(t)|\le \max\{|s(t_0+)|,|s(t_0-)|\}\le 2\e$
  and the argument in \eqref{Ke1} gives
  $$
     d(\psi_{t+\si(t)}(y),\psi_t(x_i))\le \e +  \lV\partial_t\psi\rV_{\sup}  2\e
   = :K_2\,\e \qquad \text{ when }\quad |t-t_0|\le3\e.
  $$

   \end{enumerate}
   \end{Remark}
   \bigskip
  
  \begin{Theorem}[Bowen~\cite{Bowen6} Thm. (2.2) p. 6]\label{SHL}
  \quad
  
  Given $L>0$ there are $\de_0, \,Q>0$ such that if $0<\de<\de_0$ and
  $(T,\Ga)$ is a $\de$-possible $L$-specification on $\La$ then
  $U^*_\e(T,\Ga)\ne \emptyset$ with $\e= Q\de$.   \end{Theorem}
  
  \medskip 
  
  This Theorem is proven in Bowen~\cite{Bowen6} with a similarly
  presented statement without the estimate on $\e$. 
  In the context of Bowen~\cite{Bowen6} the
  set $\La$ is locally maximal but this is not 
  needed for Theorem~\ref{SHL}. 
  A proof of this theorem for flows without the local maximality hypothesis
  and with the explicit estimate on $\e$ appears in Palmer \cite{Palmer}
  Theorem~9.3, p. 188. In \cite{Palmer}, \cite{Palmer2009} the theorem 
  requires an upper bound on the lengths of the intervals in $T$. 
  This is because there the theorem is proven also for perturbations 
  of the flow. Indeed by Proposition~\ref{B16} longer intervals in 
  $T$  {\sl improve} the estimate on $\e$.
  
  \bigskip
  
  \begin{Remark}\quad
  \begin{enumerate}[(a)]
  \item Theorem~\ref{SHL} does not require the local maximality of $\La$. 
  \item Without the local maximality the shadowing orbit may not be in $\La$.
  \item In Palmer~\cite{Palmer} Theorem~9.3, p.~188 there is a proof for this
           Theorem where a specification for $\phi$ in $\La$ is shadowed by
           a perturbation $\psi$ of the flow. It requires an upper bound in the 
           lengths of the intervals in $T$ and the estimate is $\e=M(\de+\si)$,
           where $\si$ is the $C^1$ distance of their vector fields.
  \item  It is possible to shadow specifications which are in a neighbourhood
            of $\La$. Namely, given $\e>0$, $L>0$ there is $\de>0$ and a 
            neighbourhood $U(\La)$ of $\La$ such that  if $(T,\Ga)$ is a 
            $\de$-possible $L$-specification on $U(\La)$  then 
            $U^*_\e(T,\Ga)\ne\emptyset$.
  \item If $y\in U_\e(s,T,\Ga)$, by Remark~\ref{RSHL},  $s(t)$ can be replaced 
           by a continuous function
           satisfying~\eqref{csh2}  with $\be=\e K_2$ and such that $t\mapsto t+s(t)$ 
           is strictly increasing and $s(t_0)=0$. By Corollary~\ref{B71}, $|s(t)|\le \e K_3$ for some 
           $K_3>0$. 
  \item  If the specification is periodic  with period $T$ and $y\in U_\e(s,T,\Ga)$, 
           with $\si(t):=t+s(t)$ a homeomorphism we have that 
           $d(\psi_{\si(t)}(y),\psi_{\si(t+T)}(y))\le 2\e$, $\forall t\in\re$.
           By the flow expansivity of $\psi$ in $\La$ (Remark~\ref{rue}),
           if $\e$ is small enough then there is $\tau\in\re$ with $\psi_\tau(y)=y$.
           Then $y$ is a periodic point.
  
          \end{enumerate}
  \end{Remark}
  
   Therefore we get           
   
   \newpage

  \begin{Corollary}\label{CSH}\quad
  
    Given $\ell>0$ there are $\de_0=\de_0(\ell)>0$ and $Q=Q(\ell)>0$ 
    such that if $0<\de<\de_0$ and
  $(T,\Ga)=(\{t_i\},\{x_i\})_{i\in\Z}$ is a $\de$-possible $\ell$-specification on $\La$ then
  there exist $y\in M$ and $\si:\re\to\re$ continuous, piecewise linear,  strictly increasing with 
  $\si(t_0)=t_0$ and $|\si(t)-t| < Q\,\de$  such that
  $$
  \forall i\in \Z \quad
   \forall t\in]t_i,t_{i+1}[
   \qquad
  d\big(\psi_{\si(t)}(y),\psi_t(x_i)\big)< Q\,\de.
  $$
  Moreover, if the specification is periodic then $y$ is a periodic point for $\phi$.
   \end{Corollary}
   
   \bigskip

    \section{Structural Stability.}\label{asst}

    As we shall see the structural stability result presented here
    does not need the local maximality of the hyperbolic set.
    
    Let $M$ be a $C^\infty$ compact manifold     and $\phi$ a $C^k$ flow on
    $M$. Let $\La$ be a hyperbolic set for $\phi$.
    Define
    \begin{align*}
    C^\a(\La,M) &:=\big\{\, u:\La\to M\;\big|\;
    u \text{ is $\a$-H\"older continuous }\}.
    \end{align*}
    This space has the structure of a Banach manifold
    modelled by the Banach space $C^\a(\La,\re^n)$ 
    with the norm $\lV f\rV:=\lV f\rV_0+\lV f\rV_\a$,
    where for $f\in C^\a(\La,\re^n)$,
    \begin{align*}
    \lV f\rV_0:=\sup_{x\in\La}|f(x)|,
    \qquad
    \lV f\rV_\a:=\sup_{x\ne y}\frac{|f(x)-f(y)|}{d(x,y)^\a}.
    \end{align*}

   Let
    \begin{align*}
    C^0_\phi(\La,M)&:=\big\{\,u\in C^0(\La, M)\;\big|\;
    D_\phi u(x):=\tfrac{d\,}{dt}u(\phi_t(x))\big\vert_{t=0}
    \text{ exists }\big\},
    \\
    C^\a_\phi(\La,M) &:=\big\{\,u\in C^\a(\La,M)\;\big|\;
    D_\phi u(x):=\tfrac{d\,}{dt}u(\phi_t(x))\big\vert_{t=0}
    \text{ exists and is $\a$-H\"older }\}.
    \end{align*}
    
    Let $X$ be the vector field of $\phi$. 
    The structural stability of the hyperbolic set $\La$ can be
    written as a solution $(u,\ga)\in C^0_\phi(\La,M)\times C^0(\La,\re^+)$ 
    to the equation
    $$
    Y\circ u = \ga \;D_\phi u
    $$
    for a vector field $Y$ nearby $X$.
    Here $u$ is the topological equivalence and $\ga$ encodes
    the reparametrization of the flow. The following theorem
    says that such solutions can be obtained as implicit
    functions of $Y$.
    
    \medskip
    
    \pagebreak
    
    \begin{Theorem}\label{Fss}\quad
    
    Let $M$ be a $C^{k+1}$ compact manifold and $\phi$ a flow of a $C^k$ vector field $X$ on  $M$.
    Let $\fX^k$ be the Banach manifold  of $C^k$ vector fields on $M$.
    Suppose that $\La$ is a hyperbolic set for $\phi_t$. 
    Then 
    \begin{enumerate}[(a)]
    \item\label{Fssa}
     There exist $0<\be<1$, a neighbourhood $\cU\subset \fX^k(M)$ of $X$
    and $C^{k-1}$ maps 
    \linebreak
    $\cU\to C_\phi^\be(\La,M):$
    $Y\mapsto u_Y$ and $\cU\to C^\be(\La,\re^+):$ $Y\mapsto \ga_Y$
    such that 
    \begin{equation}\label{Fssa1}
    Y\circ u_Y = \ga_Y\,D_\phi u_Y.
    \end{equation}
    \item\label{Fssb} The maps $\cU\to C_\phi^0(\La,M):$
    $Y\mapsto u_Y$ and $\cU\to C^0(\La,\re^+):$ $Y\mapsto \ga_Y$
    are $C^k$.
     \end{enumerate}
    \end{Theorem}
     
     The version for H\"older maps in item~\eqref{Fssa} is useful for 
     proving smooth dependence of equilibrium states, entropies and
     SBR measures, see~\cite{regu}. 
     For $Y$ near $X$ the topological equivalence $u_Y$ is uniquely
     determined if we require that
     $u_Y(x)\in\exp_x(X(x)^\perp)$
     and $u_Y$ near the identity. 
     One can change $\Ga(x)=\exp_x(X(x)^\perp\cap B(0,\de))$
     by any other smooth family of local transversal sections to the flow.
     
     The first version of Theorem~\ref{Fss}
     appears in de la Llave, Marco, Moriy\'on~\cite{LlMM}.
     Item~\eqref{Fssa} is proven in~\cite[p.~591]{KKPW2} and item~\eqref{Fssb}
     is proven in~\cite[p.~23ff.]{KKW} in the case $k=1$,
      but the proof can be immediately
     generalized to and arbitrary positive integer $k$.

     \begin{Corollary}\label{Css}\quad
     
     There is a neighbourhood $\cU\subset\fX^k(M)$ of $X$
     such that for every $Y\in\cU$ the map
      $u_Y$
     is a homeomorphism $u_Y:\La\to u_Y(\La)$.
     \end{Corollary}
     \begin{proof}\quad

     Since $\La$ is compact and $u_Y$ is continuous, it is enough
     (cf. Rudin~\cite{rudin2} Theorem~4.17)
     to prove that $u_Y$ is injective for $Y$ near $X$.

     Let $\eta>0$ be such that every periodic orbit in $\La$ has period
     larger than $4\eta$.
     Let $\a=\ov\a(\eta)<\eta$ be a flow expansivity constant for $(\La,\phi_t)$ as in 
     Definition~\ref{dfe}.  There is a neighbourhood $\cU_0$ of $X$ such that
     for all $Y\in\cU_0$ every periodic orbit
     in $B(\La,\a):=$ $\{\, y\in M\,|\, d(y,\La)<\a\,\}$ 
     has 
     \begin{equation}\label{pereta}
     \text{ period larger than $2\eta$.}
     \end{equation}
     Since $u_X=id_\La$ and $\ga_X\equiv 1$, there is a neighbourhood 
     $\cU_1\subset\cU_0$  of $X$ such that 
     \begin{equation}\label{Yga}
      \forall Y\in\cU_1\quad \forall x\in \La \qquad 
     d\big(u_Y(x),x\big)<\tfrac 12 \a 
     \quad \text{ and }\quad
     \tfrac 12 < \ga_Y(x)< 2.
     \end{equation}

      Denote by $\psi^Y_s$  the flow of $Y\in\cU_1$.
      From equation~\eqref{Fssa1} we get that for $Y\in\cU_1$, $x\in\La$ and $t\in\re$ 
      we have that
      $$
      Y(\phi_t(x)) = \ga_Y(\phi_t(x))\, \tfrac{d\,}{ds} u_Y(\phi_s(x))\big\vert_{s=t}\,.
      $$
      This implies that the equation
      \begin{equation}\label{uypsiy}
      u_Y(\phi_t(x)) = \psi^Y_{s(t)}(u_Y(x))
      \end{equation}
      has a solution $s(t)$ with $s(0)=0$ satisfying
      \begin{equation}\label{sdga}
      \tfrac{ds}{dt}(t)= \ga_Y(\phi_t(x))^{-1}.
      \end{equation}
      Since $\tfrac{ds}{dt}>0$ we have that $s(t)$ has a continuous inverse 
      $t(s):\re\hookleftarrow$
      satisfying $t(0)=0$. Also
      $$
      \forall s\in\re\qquad
      u_Y(\phi_{t(s)}(x))=\psi^Y_s(u_Y(x)),
      $$
      and using~\eqref{Yga},
      $$
      \forall Y\in\cU, \quad \forall x\in\La,\qquad \forall s\in\re \qquad \quad
      d\big(\psi^Y_s(u_Y(x)),\phi_{t(s)}(x)\big)<\tfrac 12\,\a. 
      $$

     Suppose that $Y\in\cU_1$ and  $x,\,y\in \La$ are such that $u_Y(x)=u_Y(y)$.
     There are increasing homeomorphisms $t_1,\,t_2:(\re,0)\hookleftarrow$
     such that
     $$
      d\big(\psi^Y_s(u_Y(x)),\phi_{t_1(s)}(x)\big)<\tfrac 12\,\a
      \quad\text{ and }\quad
       d\big(\psi^Y_s(u_Y(y)),\phi_{t_2(s)}(y)\big)<\tfrac 12\,\a.
     $$
     Since $u_Y(x)=u_Y(y)$ we get that
     $$
     \forall \tau\in\re \qquad
     d(\phi_\tau(x),\phi_{t_2\circ t_1^{-1}(\tau)}(y))<\a.
     $$
     Since $t_2\circ t_1^{-1}$ is continuous and $t_2\circ t_1^{-1}(0)=0$,
     by the flow expansivity of $\phi_\tau$ we have that 
     $y=\phi_v(x)$ with $|v|<\eta$. Suppose that $v\ne 0$.
     Since by~\eqref{Yga}, $\ga_Y^{-1}<2$,
     the orbit segment from $x$ to $\phi_v(x)$ is sent by $u_Y$ to a closed
     orbit $\psi^Y_s(y)$ with a period smaller than $2 \eta$.
     This contradicts the choice of $\cU_1$ in~\eqref{pereta} and~\eqref{Yga}.
     Therefore $v=0$ and hence $y=x$.

     \end{proof}

     \begin{Proposition}\label{PSSLM}\quad
     
     If $k\in\na^+$, $\La\subset M$ is a hyperbolic set for the flow $\phi$ on $M$ 
     with vector field $X$
     and $V$ is an open neighbourhood of $\La$, 
     then there is an open set $U$ 
     such that $\La\subset U\subset \ov{U}\subset V$, 
     an open set  $X\in\cU\subset \fX^k(M)$,
     a subshift of finite type $\si:\Om\to\Om$,
     $0<\be<1$ and $C^{k-1}$ maps
     $\tau:\cU\to C^\be(\Om,\re^+)$, $Y\mapsto \tau_Y$
     and
     $\pi:\cU\to C^\be(\Om,M)$, $Y\mapsto \pi_Y$
     such that the natural extension
     of $\pi_Y$ to $\pi_Y:S(\Om,\tau_Y)\to M$
     is a well defined time preserving semiconjugacy
     \begin{equation}\label{SOmtauy}
     \begin{CD}
     S(\Om,\tau_Y) @> S_t >> S(\Om,\tau_Y)
     \\
     @V \pi_Y VV @VV \pi_Y V
     \\
     \bLa_Y @> \psi^Y_t >> \bLa_Y
     \end{CD}
     \end{equation}
     between the suspended flow $S_t$ of $\si$ in $S(\Om,\tau_Y)$
     and a hyperbolic set $\bLa_Y$ for the flow $\psi^Y_t$ of $Y$
     which satisfies
     \begin{equation}\label{maxinu}
     \forall Y\in\cU\qquad 
     \textstyle\bigcap\limits_{t\in\re}\psi^Y_t(\ov U)\subset\bLa_Y\subset V.
     \end{equation}
     In particular $\La\subset \bLa_X$.

     \end{Proposition}

        \begin{proof}\quad
     
     Let $\e_0>0$ be such that 
     $$
      B(\La,\e_0):=\{\,y\in M\;|\; d(y,x)<\e_0\,\}\subset V.
     $$
     
     Let $\cU_0\subset \fX^k(M)$ be the neighbourhood of $X$ given by 
     Theorem~\ref{Fss} and Corollary~\ref{Css}.
     
     Using Corollary~\ref{Rfe} with $\eta=1$ there exist an open set
     $\cU_1\subset\fX^k(M)$ with
      $X\subset\cU_1\subset\cU_0$ and $\e_1,\,\a\in\re$ such that 
      \begin{equation}\label{a<e1}
      0<\a<\e_1<\e_0\, ,
      \end{equation}
      $\La_1^Y:=  \bigcap_{t\in\re}\psi^Y_t(\ov{B(\La,\e_1)})$ is hyperbolic
      and
     if $Y\in\cU_1$, $z,w\in \La_1^Y$, 
     $\be\in C^0(\re,\re)$, $\be(0)=0$ 
     satisfy
     \begin{equation}\label{cfa}
     \forall t\in\re\qquad d(\psi^Y_{\be(t)}(w),\phi^Y_t(z))<\a,
     \end{equation}
     then $w=\psi^Y_\xi(z)$ for some $|\xi|<1$.

     Let 
     \begin{equation}\label{defde0}
      0<\de_0<\a
     \end{equation}
      be such that any $\de_0$ possible $1$-specification for
     $(\La,\phi_t)$ is $\tfrac 13 \a$-shadowed as in the Shadowing Corollary~\ref{CSH}.
      For a metric space $(B,d)$ and $f,\,g:A\to B$ write
     $$
     d_0(f,g):=\sup_{a\in A}d(f(a),f(b)).
     $$
     
     Let $\e_2$ be such that 
     \begin{equation}\label{e2a}
     0<\a<\e_2<\e_1.
     \end{equation}

     Let $0<\e_3<\e_2$ and $X\in \cU_2\subset\cU_1$ be such that
     \begin{gather}
     \Big(\sup_{t\in[0,1]}\Lip(\phi_t)\Big) \,\e_3 +
     \sup_{Y\in\cU_2}\sup_{t\in[0,1]}d_{0}(\phi_t,\psi^Y_t)
     <\frac{\de_0}3< \frac \a 3
     \label{lipesup}
     \qquad\text{and}
     \\
     \tfrac1 3 \de_0 + \e_3 < \tfrac 23 \de_0.
     \label{de0e1de0}
     \end{gather}
     
     Observe that by the choice of $\e_1$ in~\eqref{a<e1} and Corollary~\ref{Rfe} the set
     $$
     \La_1:=\textstyle
     \bigcap_{t\in\re}\phi_t(\ov{B(\La,\e_1)})
     $$
     is hyperbolic.
     Let $X\in\cU_3\subset\cU_2$ be a neighbourhood of $X$
     from Theorem~\ref{Fss} applied to $\La_1$  so that $u_Y$
     and $\ga_Y$ are defined on $\La_1$ for all $Y\in\cU_3$;
     and that $\cU_3$ is small enough so that
      \begin{equation}\label{d0ui}
     \forall Y\in\cU_3\qquad 
     d_0(u_Y|_{\La_1},id|_{\La_1})< \tfrac 13 \a.
     \end{equation}

     Define
     \begin{equation}\label{defla2}
     \La_2:=\textstyle\bigcap_{t\in\re}\phi_t(\ov{B(\La,\e_2)}).
     \end{equation}

     By the existence of Markov Partitions for the flow $\phi_\La$,
          see Fisher-Hasselblatt~\cite[6.6.5]{FH},
     there is a hyperbolic set $\bLa$ for $\phi_t$ such that 
     $\La\subset\La_2\subset\bLa\subset B(\La,\e_1)\subset V$ and with
     a Markov partition which is the image by a time preserving Lipschitz
     semiconjugacy $\pi$ of a suspension of a subshift of finite type 
     $\si:\Om\to\Om$
     with H\"older ceiling function $\tau:\Om\to\re^+$.
     \begin{equation*}
     \begin{CD}
     S(\Om,\tau) @> S_t >> S(\Om,\tau)
     \\
     @V \pi VV @VV \pi V
     \\
     \bLa @> \phi_t >> \bLa
     \end{CD}
     \end{equation*}
     
     Observe that the definition~5.1.1 of a hyperbolic set $\La$ (or $\La_2$) in Fisher-Hasselblatt~\cite{FH} 
     does not require that the local maximality of the set $\La$ which we may not have. 
     This was a long standing problem solved by Fisher. See section~6.7 page~371 of \cite{FH}.
     
     Since $\bLa$ is invariant by $\phi_t$ and $\bLa\subset B(\La,\e_1)$
     we have that $\bLa\subset\La_1$. In particular
     for all $Y\in\cU_3$ the functions $u_Y$ and $\ga_Y$
     are defined on $\La_1\supset\bLa$.
     For  $Y\in\cU_3$ let $\bLa_Y:=u_Y(\bLa)$ and 
     $\pi_Y:=u_Y\circ \pi|_{\Om\times\{0\}}\in C^\be(\Om\times\{0\},\bLa_Y)$.
     Since $u_Y:\bLa\to \bLa_Y$ is a (non time preserving) topological equivalence
     among the flows $(\bLa,\phi_t)$ and $(\bLa_Y,\psi^Y_s)$ we have that the
     following diagram commutes:
     $$
     \begin{CD}
     \Om\times\{0\} @> \si>> \Om\times\{0\}
     \\
     @V \pi VV @VV \pi V
     \\
     \bLa @> F >> \bLa
     \\
     @V u_Y VV @VV u_Y V
       \\
       \bLa_Y @> F_Y >> \bLa_Y
     \end{CD}
     $$
     where $F_Y:= u_Y\circ F\circ u_Y^{-1}$ and $F$ is the first return map
     to the Markov partition for $\bLa$, $F(\pi\ow)=\phi_{\tau(\ow)}(\pi\ow)$.
     
     Since $u_Y$ satisfies the equalities~\eqref{uypsiy} and~\eqref{sdga} 
     we have that
     \begin{align*}
     F_Y(\pi_Y\ow) &= \psi^Y_{\tau_Y(\ow)}(\pi_Y(\ow)), \quad \text{where}
     \\
     \tau_Y(\ow) :&=\int_0^{\tau(\ow)}\ga_Y(\phi_t(\pi \ow))^{-1}\,dt.
     \end{align*}
     Then the diagram~\eqref{SOmtauy} commutes.

     It remains to prove the inclusions in~\eqref{maxinu}.
     Since $Y\mapsto u_Y\in C^0(\La,M)$ is continuous,
     and $u_X(\bLa)=\bLa\subset V$,
     there is a neighbourhood $X\in \cU_4\subset\cU_3$ such that 
     $$
     \forall Y\in\cU_4\qquad \bLa_Y=u_Y(\bLa)\subset V.
     $$
     
     Let $U:=B(\La,\e_3)$.
     Given
     $
     z\in\textstyle \bigcap_{s\in\re}\psi^Y_s(\ov U)
     $
     let $z_n:=\psi^Y_n(z)$, $n\in\Z$.
     Since 
     \linebreak
     $z_n\in \ov U=\ov{B(\La,\e_3)}$ there is
     $y_n\in \La$ such that 
     $d(y_n,z_n)\le \e_3$. 
      Define a $1$-specification $f(t)$ 
      (cf. definition~\ref{B8}) for $(\La,\phi_t)$
      by 
      $$
      f(n+t):=\phi_t(y_n), \quad n\in\Z, \quad t\in[0,1[.
      $$
      Using~\eqref{lipesup}
      we have that for $t <1$,
      \begin{align}\label{fphit}
      d(f(n+t),\psi^Y_t(z_n))&=d(\phi_t(y_n),\psi^Y_t(z_n))
      \notag\\
      &\le d(\phi_t(y_n),\phi_t(z_n))+d(\phi_t(z_n),\psi^Y_t(z_n))
      \notag\\
      &\le  \Big(\sup_{t\in[0,1]}\Lip(\phi_t) \Big) \,d(z_n,y_n)+
      \sup_{t\in[0,1]}d_{0}(\phi_t,\psi^Y_t)
      \notag\\
      &\le \tfrac 13 \de_0.
      \end{align}
      Observe  that $f$ is $\de_0$-possible because,
      using that $\psi^Y_1(z_n)=z_{n+1}$ and~\eqref{de0e1de0}, we have
      that
      $$
      d(f(n+1^-),y_{n+1})\le d(f(n+1^-),\psi^Y_1(z_n))
      +d(z_{n+1},y_{n+1})
      \le \tfrac13 \de_0 +\e_3
      < \tfrac 23 \de_0.
      $$
      By the Shadowing Corollary~\ref{CSH} and the choice of $\de_0$ in~\eqref{defde0},
       there is  $x\in M$ and 
       an increasing homeomorphism $\be:(\re,0)\to(\re,0)$ 
        such that  
      \begin{equation}\label{dbef}
      \forall t\in\re\qquad d(\phi_{\be(t)}(x),f(t^\pm))<\tfrac 13 \a.
      \end{equation}
      Since by definition $f(t)\in\La$ and $\be(t)$ is a homeomorphism,
      and using~\eqref{e2a} and \eqref{defla2}
       we obtain  that
      \begin{equation}\label{xinsatu}
      x\in \textstyle\bigcap_{\be\in\re}\phi_\be\big(B(\La,\tfrac 13 \a)\big)
      \subset\bigcap_{t\in\re}\phi_t(\ov{B(\La,\e_2)})=\La_2\subset\bLa.
      \end{equation}
      
      Then by~\eqref{dbef}, ~\eqref{fphit} and \eqref{defde0},
      \begin{align}
      d(\phi_{\be(n+t)}(x),\psi^Y_{n+t}(z) )
      &=
      d(\phi_{\be(n+t)}(x),\psi^Y_t(z_n))
      \notag\\
      &\le d(\phi_{\be(n+t)}(x),f(n+t))+d(f(n+t),\psi^Y_t(z_n))
      \notag\\
      &\le \tfrac 13 \a+ \tfrac 13 \de_0\le \tfrac 23 \a.
      \label{le23a}
       \end{align}

      There is a homeomorphism $\si:(\re,0)\to(\re,0)$ such that
      \begin{equation}\label{uYsi}
      	\psi^Y_{\si(t)}(u_Y(x))=u_Y(\phi_{\be(t)}(x)).
      \end{equation}
      Indeed, comparing~\eqref{uYsi} and~\eqref{Fssa1} we get that
      $$
      \si(0)=0 \quad\text{and}\quad \dot\si(t)=\dot\be(t)\,\ga_Y(\phi_{\be(t)}(x))^{-1}.
      $$
      Therefore, using~\eqref{d0ui} and \eqref{le23a},
      \begin{align*}
      d(\psi^Y_{\si(t)}(u_Y(x)),\psi^Y_t(z))
      &= d(u_Y(\phi_{\be(t)}(x)),\psi^Y_t(z))
      \\
      &\le 
      d(u_Y(\phi_{\be(t)}(x)),\phi_{\be(t)}(x))
      +d(\phi_{\be(t)}(x),\psi^Y_t(z))
      \\
      &\le d_0(u_Y,id)+  \tfrac 23 \a <\a.
      \end{align*}
      
      By the choice of $\a$ in~\eqref{cfa}
       we have that there is $\xi\in\re$ such that
      $z=\psi^Y_\xi(u_Y(x))$.
      Since $\bLa_Y=u_Y(\bLa)$ is $\psi^Y_t$-invariant 
      and by~\eqref{xinsatu}, $x\in\bLa$,
      we obtain that $z\in\bLa_Y$. Therefore
      $$
      \textstyle\bigcap_{s\in\re}\psi^Y_s(\ov U)\subset\bLa_Y.
      $$

        \end{proof}

     \begin{Proposition}\label{liftmuSSLM}\quad
     
     Let $\phi\in\cU$ and $\La\subset U\subset V$ be from Proposition~\ref{PSSLM}.
     If $Y\in\cU$ and $\mu$ is a $\psi_t^Y$-invariant Borel probability with 
     $\supp\mu\subset \ov U$ then there is an $S_t$-invariant Borel probability
     on $S(\Om,\tau_Y)$ such that $(\pi_Y)_*(\nu)=\mu$.
     \end{Proposition}
     
     \begin{proof}
     By \eqref{maxinu} in Proposition~\ref{PSSLM} $\supp(\mu)\subset \La_Y$.
     For $f\in C^0(\La_Y,\re)$ let  $G(f\circ\pi_Y)=\int f\,d\mu$.
     Then $G$ defines a positive linear functional on a subspace $W$
     of $C^0(S(\Om,\tau_Y),\re)$. By the Riesz extension theorem, 
     $G$ extends to a positive linear functional on $C^0(S(\Om,\tau_Y),\re)$.
     Since $G(1)=G(1\circ \pi)=1$, the extension $G$ corresponds to a Borel probability 
     $\be$ 
     on $S(\Om,\tau_Y)$. 
     By the construction $(\pi_Y)_*\be=\mu$.
     By compactness we can choose a sequence 
     $T_k\to\infty$ such that 
     $ \nu=\lim_k\tfrac 1{T_k}\int_0^{T_k}(S_t)_*\be\, dt$
     exists. We have that $\nu$ is $S_t$-invariant and using that
     $$
     (\pi_Y)_* (S_t)_*\be =(\psi^Y_t)_* (\pi_Y)_*\be=(\psi^Y_t)_* \mu=\mu,
     $$
     we get that $(\pi_Y)_*\nu=\mu$.
     
     \end{proof}

    \pagebreak

      \section{Stability of hyperbolic Ma\~n\'e sets.}\label{ashms}

      Given a Tonelli lagrangian $L:TM\to\re$ let $E_L:=v\cdot L_v-L$
      be its energy function and define
      $$
      e_0(L):= \inf\{\, k\in\re\;|\; \pi(E_L^{-1}\{k\})=M\,\}.
      $$
      Observe that
      $$
      \tfrac{d\,}{dt}E_L(x,tv)=v\cdot L_{vv}(x,tv)\cdot v>0 \qquad\text{if }v\ne 0.
      $$
      Then $f(t):=E_L(x,tv)$ is increasing on $t>0$.
      This implies that for $k>e_0(L)$ the radial projection is a diffeomorphism
      between the unit tangent bundle and the energy level $E^{-1}_L\{k\}$.

      Denote by  $O_M$ the zero section of $TM$.
      
      \begin{Lemma}\label{LSH1}\quad
      
      If $E^{-1}\{c(L)\}\cap O_M\ne \emptyset$ then 
      $c(L)=e_0(L)$ and $E^{-1}_L(e_0)\cap O_M\subset \cA(L)\subset\mN(L)$.
      \end{Lemma}
      \begin{proof}\quad
      
      Let $$
      \Si^+(L):=\big\{v\in TM\;\big|\; \pi\circ \phi_t(v)|_{[0,+\infty[}\text{ is semi-static }\big\}.
      $$
      We have that $E(\Si^+(L))=\{c(L)\}$ (Ma\~n\'e~\cite[p.~146]{Ma7}).
      By the Covering Property (Ma\~n\'e~\cite[Theorem~VII]{Ma7} 
      also~\cite[Th.~VII]{CDI}), $\pi(\Si^+(L))=M$. Therefore $c(L)\ge e_0(L)$ and
      $$
      E^{-1}\{c(L)\}\cap O_M\ne \emptyset \qquad\then\qquad
      c(L)=e_0(L).
      $$
      Let 
      $$
      F(x,v):=L(x,v)+E_L(x,v)=v\cdot L_v(x,v).
      $$
      If $\ga:[0,T]\to M$ is a closed curve with energy $c(L)$ we have that
      \begin{align}\label{intFcA}
      \int_0^T F(\ga,\dga)\,dt =
      \int_0^T \big(L(\ga,\dga)+c(L)\big)\,dt
     \ge A_{L+c}(\ga|_{[0,t]})+\Phi_c(\ga(t),\ga(0))\ge 0
     \quad \forall t\in]0,T[.
      \end{align}
      Suppose that $(x_0,0)\in E^{-1}_L\{c(L)\}\cap O_M\ne \emptyset$.
      Let $\eta(t)\equiv x_0$ for $t\in[0,2]$.
      Then $E_L(\eta,\deta)=E_L(x_0,0)=c(L)$,
      $F(x_0,0)=0$ and $\int_0^2F(\eta,\deta)\, dt =0$.
      From~\eqref{intFcA} we obtain that $\eta$ is static and hence
      $(\ga,\dga)=(x_0,0)\in\cA(L)$.

     \end{proof}
     
     \begin{Corollary}\label{CSH1}\quad 
     
     If $\mN(L)$ does not contain fixed points of the lagrangian flow
     then $c(L)>e_0(L)$ and the energy level 
     $E_L^{-1}\{c(L)\}$ is diffeomorphic to the unit 
     tangent bundle $SM$ under the radial projection. 
     \end{Corollary}

     \medskip

     \begin{Theorem}\label{SDL}\quad
     
     Suppose that $L:TM\to\re$ is a Tonelli lagrangian and that 
     its Ma\~n\'e set $\mN(L)$ is hyperbolic without fixed points. 
     Let $V$ be an open set with $\mN(L)\subset V$.
     Then there is a subshift
     of finite type $\si:\Om\to\Om$ and there are
     open sets $\mN(L)\subset U\subset V$ and $0\in\cU\subset C^2(M,\re)$
     and continuous maps 
     $C^2(M,\re)\supset\cU\ni \phi\mapsto \tau_\phi\in C^0(\Om,\re^+)$ and
     $C^2(M,\re)\supset\cU\ni \phi\mapsto \pi_\phi\in C^0(S(\Om,\tau_\phi),TM)$, 
     where $\big(S(\Om,\tau_\phi),S_t\big)$ is the suspension flow of $\si$ 
     with ceiling function $\tau_\phi$,
     and there are hyperbolic sets $\La_\phi=\pi_\phi(S(\Om,\tau_\phi))$ for the flow
     $\vr^{L+\phi}_t$ of $L+\phi$ restricted to the energy level 
     $E_{L+\phi}^{-1}\{c(L+\phi)\}$ such that
     $$
     \forall \phi\in\cU\qquad
     \mN(L+\phi)\subset 
     \textstyle\bigcap_{t\in\re}\vr^{L+\phi}_t(U)\subset \La_\phi \subset V
     $$
     and the following diagram commutes for all $t\in\re$:
     $$
     \begin{CD}
     S(\Om,\tau_\phi) @> S_t>> S(\Om,\tau_\phi)
     \\
     @V \pi_\phi VV @VV \pi_\phi V
     \\
     \La_\phi @> \vr^{L+\phi}_t >> \La_\phi
     \end{CD}
     $$
     
     Moreover any invariant measure $\mu$ for $L+\phi$ with $\supp(\mu)\subset U$
     lifts to an invariant measure $\nu$ on $S(\Om,\tau_\phi)$ 
     with $(\pi_\phi)_*\,\nu=\mu$. In particular $\mu$ is a 
     minimizing measure for 
     $L+\phi$ iff it is the projection of an invariant probability $\nu$ on 
     $S(\Om,\tau_\phi)$ which minimizes the integral of the function 
     $A_\phi:=(L+\phi)\circ\pi_\phi$.
    
     \end{Theorem}

      \begin{proof}\quad
   
     By Lemma~5.1 in~\cite{CP} for all $\ell\ge 2$ the map 
     $C^\ell(M,\re)\ni\phi\mapsto c(L+\phi)\in\re$ is continuous.
     And by Lemma~5.2 in~\cite{CP} the map
     $C^\ell(M,\re)\ni\phi\mapsto \mN(L+\phi)$ is upper semicontinuous.
          By Corollary~\ref{CSH1} for $\phi\in\cU$ small enough we can 
     identify the energy levels 
     $E^{-1}_{L+\phi}\{c(L+\phi)\}\approx E^{-1}_{L}\{c(L)\}\approx SM$
     with the unit tangent bundle $SM$ 
     and consider their lagrangian flows as perturbations of the
     flow of $L$ on the same manifold $SM$.
     
     Let $\I:=]c(L)-\e,c(L)+\e[$ with $\e>0$ small.
     Let $P_{\phi,k}:SM \to E^{-1}_{L+\phi}\{k\}$
     be the radial projection and let $X_\phi$ be the Lagrangian 
     vector field for $L+\phi$. Let $\fX^1(SM)$ be the
     vector space of  $C^1$ vector fields on $SM$.
     The map $\X:C^\ell(M,\re)\times \I\to \fX^1(SM)$,
     $(\phi,k)\mapsto (dP_{\phi,k})^{-1}\circ X_{\phi}\circ P_{\phi,k}$
     is $C^{\ell-2}$ in a neighbourhood of $(\phi,k)=(0,c(L))$. 
     If we compose this map with the continuous function 
     $\phi\mapsto k=c(L+\phi)$ we obtain a continuous map 
     $C^2(M,\re) \to \fX^1(SM)$,
     $\phi\mapsto \X(\phi,c(L+\phi))$.
     The flow of 
     this vector field is   
     \linebreak
     $\psi^\phi_t:=P_\phi^{-1}\circ \vr^{L+\phi}_t\circ P_\phi$,
     where $P_\phi:=P_{\phi,c(L+\phi)}$,
     which is smoothly conjugate to the lagrangian flow of $L+\phi$ on
     $E^{-1}_{L+\phi}\{c(L+\phi)\}$, and $\phi\mapsto \psi^\phi_t$ is a
     continuous family of $C^1$ flows on $SM$.
     Then there are neighbourhoods $\cU$ of $0$ and $U\subset V$ of $\mN(L)$
     in $SM$ such that for any $\phi\in\cU$ the set 
     $\bigcap_{t\in\re}\psi^{\phi}_t(\ov U)$ is hyperbolic and
     $P_\phi^{-1}(\mN(L+\phi))\subset U$, using the 
     upper semicontinuity of $\phi\mapsto \mN(L+\phi)$.

     Applying Proposition~\ref{PSSLM} and Proposition~\ref{liftmuSSLM},
     shrinking $\cU$ and $U$ if necessary, we obtain Proposition~\ref{SDL}.
     
     \end{proof}

    \begin{Remark}\label{remell}
    \quad

    In the proof of Theorem~\ref{SDL} the map 
    $\X:C^\ell(M,\re)\times \I\to \fX^{\ell-1}(SM)$
    is conitnuous and then the map 
    $C^\ell(M,\re)\to \fX^{\ell-1}(SM)$, 
    $\phi\mapsto \X(\phi,c(L+\phi))$ is 
    continuous in the $C^{\ell-1}$ topology
    for $\fX^{\ell-1}(SM)$.

    \end{Remark}


\begin{thebibliography}{10}

\bibitem{Be3}
Patrick Bernard, \emph{Existence of {$C^{1,1}$} critical sub-solutions of the
  {H}amilton-{J}acobi equation on compact manifolds}, Ann. Sci. \'Ecole Norm.
  Sup. (4) \textbf{40} (2007), no.~3, 445--452, See also arXiv:math/0512018.

\bibitem{Bowen10}
Rufus Bowen, \emph{Entropy for group endomorphisms and homogeneous spaces},
  Trans. Amer. Math. Soc. \textbf{153} (1971), 401--414.

\bibitem{Bowen9}
\bysame, \emph{Entropy-expansive maps}, Trans. Amer. Math. Soc. \textbf{164}
  (1972), 323--331.

\bibitem{Bowen6}
\bysame, \emph{Periodic orbits for hyperbolic flows}, Amer. J. Math.
  \textbf{94} (1972), 1--30.

\bibitem{Bowen3}
\bysame, \emph{Symbolic dynamics for hyperbolic flows}, Amer. J. Math.
  \textbf{95} (1973), 429--460.

\bibitem{BoWa}
Rufus Bowen and Peter Walters, \emph{Expansive one-parameter flows}, J.
  Differential Equations \textbf{12} (1972), 180--193.

\bibitem{BQ}
Xavier Bressaud and Anthony Quas, \emph{Rate of approximation of minimizing
  measures}, Nonlinearity \textbf{20} (2007), no.~4, 845--853.

\bibitem{Carneiro}
Mario Jorge~Dias Carneiro, \emph{On minimizing measures of the action of
  autonomous {L}agrangians}, Nonlinearity \textbf{8} (1995), no.~6, 1077--1085.

\bibitem{regu}
Gonzalo Contreras, \emph{Regularity of topological and metric entropy of
  hyperbolic flows}, Math. Z. \textbf{210} (1992), no.~1, 97--111.

\bibitem{CDI}
Gonzalo Contreras, Jorge Delgado, and Renato Iturriaga, \emph{Lagrangian flows:
  the dynamics of globally minimizing orbits. {I}{I}}, Bol. Soc. Brasil. Mat.
  (N.S.) \textbf{28} (1997), no.~2, 155--196, special issue in honor of Ricardo
  Ma\~n\'e.

\bibitem{ham}
Gonzalo Contreras and Renato Iturriaga, \emph{Convex {H}amiltonians without
  conjugate points}, Ergodic Theory Dynam. Systems \textbf{19} (1999), no.~4,
  901--952.

\bibitem{CILib}
\bysame, \emph{{G}lobal {M}inimizers of {A}utonomous {L}agrangians},
  $22^{\text{o}}$ Coloquio Bras. Mat., IMPA, Rio de Janeiro, 1999.

\bibitem{CP}
Gonzalo Contreras and Gabriel~P. Paternain, \emph{Connecting orbits between
  static classes for generic {L}agrangian systems}, Topology \textbf{41}
  (2002), no.~4, 645--666.

\bibitem{CroPo}
Sylvain Crovisier and Rafael Potrie, \emph{Introduction to partially hyperbolic
  dyanmics}, lecture notes for a minicourse at ICTP School on Dynamical Systems
  2015.

\bibitem{LlMM}
Rafael de~la Llave, Jos\'e~Manuel Marco, and Roberto Moriy\'on, \emph{Canonical
  perturbation theory of anosov systems and regularity results for the licsic
  cohomology equation}, Ann. of Math. (2) \textbf{123} (1986), no.~3, 537--611.

\bibitem{Falconer0}
Kenneth Falconer, \emph{Fractal geometry}, second ed., John Wiley \& Sons,
  Inc., Hoboken, NJ, 2003, Mathematical foundations and applications.

\bibitem{Fa7}
Albert Fathi, \emph{Expansiveness, hyperbolicity and {H}ausdorff dimension},
  Comm. Math. Phys. \textbf{126} (1989), no.~2, 249--262.

\bibitem{FaSi}
Albert Fathi and Antonio Siconolfi, \emph{Existence of {$C\sp 1$} critical
  subsolutions of the {H}amilton-{J}acobi equation}, Invent. Math. \textbf{155}
  (2004), no.~2, 363--388.

\bibitem{FH}
Todd Fisher and Boris Hasselblatt, \emph{{H}iperbolic {F}lows}, Zurich Lectures
  in Advanced Mathematics, European Mathematical Society, Berlin, 2019.

\bibitem{HK}
Boris Hasselblatt and Anatole Katok, \emph{Introduction to the modern theory of
  dynamical systems}, Cambridge University Press, Cambridge, 1995, With a
  supplementary chapter by Katok and Leonardo Mendoza.

\bibitem{HPS}
Morris~W. Hirsch, Charles~C. Pugh, and Michael Shub, \emph{Invariant
  manifolds}, Springer-Verlag, Berlin, 1977, Lecture Notes in Mathematics, Vol.
  583.

\bibitem{HuWall}
Witold Hurewicz and Henry Wallman, \emph{Dimension {T}heory}, Princeton
  Mathematical Series, v. 4, Princeton University Press, Princeton, N. J.,
  1941.

\bibitem{KKPW2}
Anatole Katok, G.~Knieper, Mark Pollicott, and Howard Weiss,
  \emph{{Differentiability and analyticity of topological entropy for Anosov
  and geodesic flows.}}, Invent. Math. \textbf{98} (1989), no.~3, 581--597.

\bibitem{KKW}
Anatole Katok, Gerhard Knieper, and Howard Weiss, \emph{{Formulas for the
  derivative and critical points of topological entropy for Anosov and geodesic
  flows.}}, Commun. Math. Phys. \textbf{138} (1991), no.~1, 19--31.

\bibitem{Ma6}
Ricardo Ma\~n{\'e}, \emph{Generic properties and problems of minimizing
  measures of {L}agrangian systems}, Nonlinearity \textbf{9} (1996), no.~2,
  273--310.

\bibitem{Ma7}
\bysame, \emph{Lagrangian flows: the dynamics of globally minimizing orbits},
  International Conference on Dynamical Systems (Montevideo, 1995), Longman,
  Harlow, 1996, Reprinted in Bol. Soc. Brasil. Mat. (N.S.) {\bf 28} (1997), no.
  2, 141--153., pp.~120--131.

\bibitem{Mat5}
{John N.} Mather, \emph{Action minimizing invariant measures for positive
  definite {L}agrangian systems}, Math. Z. \textbf{207} (1991), no.~2,
  169--207.

\bibitem{Palmer}
Kenneth~James Palmer, \emph{Shadowing in dynamical systems}, Mathematics and
  its Applications, vol. 501, Kluwer Academic Publishers, Dordrecht, 2000,
  Theory and applications. \MR{1885537}

\bibitem{Palmer2009}
\bysame, \emph{Shadowing lemma for flows}, Scholarpedia \textbf{4} (2009),
  no.~4, 7918.

\bibitem{PS}
Charles Pugh and Michael Shub, \emph{The {$\Omega$}-stability theorem for
  flows}, Invent. Math. \textbf{11} (1970), 150--158.

\bibitem{rudin2}
Walter Rudin, \emph{Principles of mathematical analysis}, third ed.,
  McGraw-Hill Book Co., New York-Auckland-D\"usseldorf, 1976, International
  Series in Pure and Applied Mathematics.

\bibitem{Smale}
Steve Smale, \emph{{Differentiable dynamical systems. I: Diffeomorphisms; II:
  Flows; III: More on flows; IV: Other Lie groups}}, Bull. Am. Math. Soc.
  \textbf{73} (1967), 747--792, 795--804, 804--808; 808--817; Appendix to I:
  Anosov diffeomorphisms by John Mather, 792--795.

\bibitem{stein}
Elias~M. Stein, \emph{Singular integrals and differentiability properties of
  functions}, Princeton Mathematical Series, No. 30, Princeton University
  Press, Princeton, N.J., 1970.

\bibitem{Walters2}
Peter Walters, \emph{Ergodic theory---introductory lectures}, Lecture Notes in
  Mathematics, Vol. 458, Springer-Verlag, Berlin-New York, 1975.

\bibitem{Walters}
\bysame, \emph{An introduction to ergodic theory}, Graduate Texts in Math. 79,
  Springer, 1982.

\bibitem{lsy1}
Lai~Sang Young, \emph{Entropy of continuous flows on compact {$2$}-manifolds},
  Topology \textbf{16} (1977), no.~4, 469--471.

\end{thebibliography}

\def\cprime{$'$} \def\cprime{$'$} \def\cprime{$'$} \def\cprime{$'$}
\providecommand{\bysame}{\leavevmode\hbox to3em{\hrulefill}\thinspace}
\providecommand{\MR}{\relax\ifhmode\unskip\space\fi MR }
\providecommand{\MRhref}[2]{%
  \href{http://www.ams.org/mathscinet-getitem?mr=#1}{#2}
}
\providecommand{\href}[2]{#2}

\end{document}